\documentclass[reqno]{amsart}
\usepackage{amsthm,amsmath,amssymb,amscd, mathrsfs,graphics, hyperref,enumitem,amsfonts,graphicx,fancyhdr,geometry,tabularx,url,array,amscd,mathdots}
\usepackage{diagrams}
\usepackage[all,cmtip]{xy} 
\usepackage{verbatim} 
\allowdisplaybreaks

\newarrow {Mapsto} |--->
\newarrow {Iso} ---->

\newcommand{\pn}{\mathbb{P}(1,n)}

\newcommand{\id}{\mathbf{1}}
\newcommand{\xmm}{x_{m_1+m_2}}
\newcommand{\tpsi}{\widetilde{\psi}}

\newcommand{\Ch}{\mathrm{Ch}}

\newcommand{\IGX}{I_GX}

\newcommand{\inCh}{\ovCh}

\newcommand{\IzeroX}[1]{I_G^{0,V}X}

\newcommand{\Smv}{^V\kern-.5em\mathscr{S}}

\newcommand{\TV}{^V\kern-.4em\mathscr{T}}

\newcommand{\age}{\operatorname{age}}

\newcommand{\cP}{\mathcal{P}}

\newcommand{\cf}{{\mathscr F}}

\newcommand{\cl}{{\mathscr L}}

\newcommand{\euler}{\operatorname{eu}}

\newcommand{\ix}{\stack{X}}

\newcommand{\lam}{\lambda}

\newcommand{\nc}{\mathbb{C}}

\newcommand{\nq}{\mathbb{Q}}

\newcommand{\oCh}{\mathscr{C}\kern-.25em{h}} 
\newcommand{\ovCh}{\widetilde{\oCh}}

\newcommand{\olam}{\widetilde{\lam}}

\newcommand{\pro}{\mathbb{P}}

\newcommand{\psit}{\widetilde{\psi}}

\newcommand{\rk}{\mathrm{rk}}

\newcommand{\stack}[1]{\mathscr {#1}}

\newcommand{\Iner}{\mathfrak{I}} 

\newcommand{\IIX}{{\Iner \kern-.75em \Iner}_{\ix}}

\DeclareMathOperator{\Td}{Td}

\newcommand{\Imult}[1]{\mathbb{I}^{#1}}

\newtheorem{thm}{Theorem}[section] 

\newtheorem{prop}[thm]{Proposition} 
\newtheorem{crl}[thm]{Corollary}

\theoremstyle{definition}
\newtheorem{rem}[thm]{Remark} 

\newtheorem{df}[thm]{Definition} 

\newtheorem{df-pr}[thm]{Definition-Proposition}

\theoremstyle{remark} 
 \renewcommand{\thenota}{\kern-1ex}

\begin{document}
\title{Adams operations on the virtual K-theory of $\mathbb{P}(1,n)$}

\author [T. Kimura]{Takashi Kimura} \address {Department of
Mathematics and Statistics; 111 Cummington Mall, Boston University;
Boston, MA 02215, USA } \email{kimura@math.bu.edu} 

\author [R. Sweet]{Ross Sweet} \address {Department of
Mathematics and Statistics; 111 Cummington Mall, Boston University;
Boston, MA 02215, USA }
\email{rsweet@math.bu.edu}

\date{\today}

\begin{abstract}  
%
%
We analyze the structure of the virtual (orbifold) K-theory ring of the complex orbifold $\pro(1,n)$ and its virtual Adams (or power) operations, by using the non-Abelian localization theorem of Edidin-Graham \cite{EG}. In particular, we identify the group of virtual line elements and obtain a natural presentation for the virtual K-theory ring in terms of these virtual line elements. This yields a surjective homomorphism from the virtual K-theory ring of $\pro(1,n)$ to the ordinary K-theory ring of a crepant resolution of the cotangent bundle of $\pro(1,n)$ which respects the Adams operations. Furthermore, there is a natural subring of the virtual K-theory ring of $\pro(1,n)$ which is isomorphic to the ordinary K-theory ring of the resolution. This generalizes the results of Edidin-Jarvis-Kimura \cite{EJK2} who proved the latter for $n=2,3$.
\end{abstract}

\maketitle \setcounter{tocdepth}{1}

\tableofcontents

\section{Introduction}
We begin by quickly recalling a few facts about the virtual K-theory ring and virtual Adams (or power) operations of a complex orbifold. We will focus upon the case of primary interest to us, $\pro(1,n)$, for simplicity of exposition, keeping in mind that these notions can be defined quite generally, e.g. for toric orbifolds. We refer the interested reader to \cite{EJK1, EJK2} for the general case.

Let $\pro(1,n)$ be the weighted projective line regarded as the toric orbifold $[X/G]$ where the smooth manifold $X = \nc^2-\{(0,0)\}$ has the action of $G$, the complex torus $\nc^*$, given by the map $G\times X\to X$ taking $(\alpha,(z_1,z_2))\mapsto (\alpha z_1,\alpha^n z_2)$. The associated inertia orbifold $I\pro(1,n)$ is the quotient orbifold $[(I_G X)/G ]$ where the inertia manifold $I_G X\subset G\times X$, consisting of all points $(m,x)$ such that $m x = x$, is a $G$-manifold with the induced action $G\times I_G X\to I_G X$ given by $(h,(g,x))\mapsto (h g h^{-1},h x)$. Notice that $I\pro(1,n)$ contains $\pro(1,n)$ via the 
the $G$-equivariant embedding $X\to I_G X$ taking $x\mapsto (1,x)$.

Let $K(I\pro(1,n))$ be the Grothendieck group of complex vector bundles on the orbifold $I\pro(1,n)$ or, equivalently, $K(I\pro(1,n))$ is $K_G(I_G X)$, the Grothendieck group of $G$-equivariant vector bundles on $I_G X$. There are many so-called \textsl{inertial products} on $K(I\pro(1,n))$ in the sense of \cite{EJK1} which make $K(I\pro(1,n))$ into a commutative, unital ring containing the ordinary K-theory ring $K(\pro(1,n))$, called the \textsl{untwisted sector}, as a subring. Recall that the motivating example of an inertial product is the \textsl{orbifold product} which, in K-theory, is due to Adem-Ruan-Zhang \cite{ARZ} and Jarvis-Kaufmann-Kimura \cite{JKK2} (and, in terms of a Lie group presentation, Edidin-Jarvis-Kimura \cite{EJK0}).  This orbifold product is a K-theoretic version of an orbifold product originally introduced in cohomology theory by Chen-Ruan \cite{CR} as the genus $0$, $3$-pointed, degree $0$ stable maps contribution to (orbifold) Gromov-Witten theory.

The \textsl{virtual (orbifold) product} is an inertial product introduced by Gonzalez-Lupercio-Segovia-Uribe-Xicotencatl \cite{GLSUX}. Furthermore, in \cite{EJK2}, it was shown that the virtual K-theory ring of any toric orbifold admits \textsl{virtual Adams (or power) operations}, a collection of ring homomorphisms $\psit^k:K(I\pro(1,n))\to K(I\pro(1,n))$ which satisfy $\psi^1=id$, and the conditions $\psit^k\circ\psit^\ell = \psit^{k\ell}$ for all $k,\ell\geq 1$, yielding the structure of a so-called \textsl{$\psi$-ring.}  In fact, the $\psi$-ring structure on $K(I\pro(1,n))$ induces a $\lam$-ring structure on $K(I\pro(1,n))_\nq$.  In addition, on the untwisted sector $K(\pro(1,n))$, these virtual Adams operations reduce to the ordinary Adams operations which satisfy 
\begin{equation}\label{eq:VirtualLineElt}
\psit^k(\cl) = \cl^k
\end{equation}
for all $k\geq 1$ when $\cl$ is the class of any line bundle.  

The previous equation motivates the following definition: An element $\cl$ in the virtual $\psi$-ring $K(I\pro(1,n))_\nc$ is said to be a \textsl{virtual line element}  if it is invertible and satisfies Equation (\ref{eq:VirtualLineElt}) for all $k\geq 1$. The multiplicative group of virtual line elements $\cP\subset K(I\pro(1,n))_\nc$ contains, in particular, classes of ordinary line bundles on $\pro(1,n)$. Virtual line elements share many of the same properties as ordinary line bundles in ordinary equivariant K-theory. For example, they possess a virtual version of an \textsl{Euler class} constructed from the virtual $\lam$-ring structure and a virtual version of dualization (see \cite{EJK2} for details).

In this paper, we apply non-Abelian localization \cite{EG} to $K(I\pro(1,n))_\nc$ and express the virtual product and virtual Adams operations in terms of natural generators on the localization. We use these generators to calculate the group of virtual line elements $\cP$, show that $\cP$ spans the vector space $K(I\pro(1,n))_\nc$, and then give a simple presentation of the virtual K-theory ring $K(I\pro(1,n))_\nc$ in terms of these virtual line elements. Finally, we show that for a particular crepant resolution, $Z_n$, of the total space of the cotangent bundle of $\pro(1,n)$, there is a surjective homomorphism from the virtual K-theory ring $K(I\pro(1,n))_\nc$ to ordinary K-theory ring $K(Z_n)_\nc$  which respects the $\psi$-ring structures in the spirit of the hyper-K\"{a}hler resolution conjecture. Indeed, there is a summand of the virtual K-theory ring $K(I\pro(1,n))_\nc$, isomorphic to the completion with respect to the augmentation ideal, which is isomorphic to $K(Z_n)_\nc$ as $\psi$-rings.   These results generalize those of \cite{EJK2} where a different presentation of the virtual K-theory ring and the same results for the crepant resolution for $n=2,3$ were obtained without using localization.

Finally, it is worth pointing out that these techniques should apply to general toric orbifolds.  It would be very interesting to study these questions in greater generality.

\subsection*{Acknowledgments}
We would like to thank Dan Edidin, Tyler Jarvis, and Tomoo Matsumura for useful discussions.
%

\section{Review of inertial and virtual K-theory}

This section provides a brief introduction to theory of inertial pairs on smooth Deligne-Mumford stacks developed in \cite{EJK1, EJK2}.  An inertial pair defines a multiplication on the K-theory and Chow ring of the inertia stack.  These rings have compatible power operations for a special case of inertial pairs, which will be of interest here.  Let $X$ be a scheme or algebraic space (or smooth manifold) and $G$ a linear algebraic group acting properly on $X$.  Then the quotient $[X/G]$ is a smooth Deligne-Mumford stack (or complex orbifold).  
\begin{df} Define the \emph{inertia space} of $X$ to be $$\IGX=\{(g,x)\,|\, gx=x\}\subset G\times X.$$ Let $X^m=\{(m,x)\,|\,mx=x\}\subseteq\IGX$.  In particular, $X^1=X$.  The inertia space has an action of the group $G$ via $h\cdot(g,x)=(hgh^{-1},x)$.  We may then define the double inertia space by $$\Imult{2}_GX=\{(g_1,g_2,x)\,|\,g_ix=x,\, i=1,2\}\subset G^2\times X.$$
\end{df}
\begin{df} Let $\mu:\Imult{2}_GX\rightarrow\IGX$ be the composition map $\mu(g_1,g_2,x)=(g_1g_2,x)$, and let $e_i:\Imult{2}_GX\rightarrow\IGX$ for $i=1,2$ be the evaluation maps $e_i(g_1,g_2,x)=(g_i,x)$.
\end{df}
An \emph{inertial product} will be defined on $K_G(\IGX)$ and $A^*_G(\IGX)$ associated to an \emph{inertial pair} $(\stack{R},\stack{S})$ \cite{EJK1}, where $\stack{R}$ is a $G$-equivariant vector bundle on $\Imult{2}_GX$ and $\stack{S}$ is a non-negative class in $K_G(\IGX)_{\mathbb{Q}}$, as we will briefly review below.  In a special case of inertial pairs associated to vector bundles, we will obtain the virtual orbifold product, which will be the focus of the remaining sections of this paper.
\begin{df} For $x,y\in K_G(\IGX)$, define the inertial product on $K_G(\IGX)$ by 
\begin{equation} \label{eq.inprod} x*_{\stack{R}}y=\mu_*(e_1^*(x)\cdot e_2^*(y)\cdot \euler(\stack{R})),
\end{equation}  where $\euler(\stack{R})$ is the Euler class of $\stack{R}$.  Recall that in K-theory, the Euler class is $\lambda_{-1}(\stack{R}^*)$ and in Chow or cohomology is the top Chern class.  The same formula gives the inertial product on $A^*_G(\IGX)$, and on $H_G^*(\IGX)$.
\end{df}
The compatibility condition on $\stack{S}$ giving an inertial pair is the identity $$\stack{R}=e_1^*(\stack{S})+e_2^*(\stack{S})-\mu^*(\stack{S})+T_{\mu},$$ where $T_{\mu}$ is the relative tangent bundle.  We refer the reader to \cite{EJK1} \S3 for additional details concerning the definition of an inertial pair.
\begin{prop}[\cite{EJK0},\S3] If $(\stack{R},\stack{S})$ is an inertial pair, then the associated inertial product $*_{\stack{R}}$ on $K_G(\IGX)$ is commutative and associative with identity $\mathbf{1}_X\in K_G(X)$.  Further, $K_G(X)\subseteq K_G(\IGX)$ is a subring where the inertial product reduces to the ordinary product on equivariant K-theory.
\end{prop}
The same result holds for Chow and cohomology.  Furthermore, there is an inertial Chern character map that gives a homomorphism between the inertial K-theory and inertial Chow rings.
\begin{prop}[\cite{EJK1},Prop 3.8] If $(\stack{R},\stack{S})$ is an inertial pair, then the inertial Chern character $$\inCh:K_G(\IGX)_{\mathbb{Q}}\longrightarrow A^*_G(\IGX)_{\mathbb{Q}},$$ given by $$\inCh(V)=\Ch(V)\cdot \Td(-\stack{S}),$$ where $\Ch(V)$ is the ordinary Chern character of $V$, $\Td(V)$ is the Todd class, and $\cdot$ is the ordinary product, is a homomorphism of rings with respect to the inertial products on $K_G(\IGX)$ and $A^*_G(\IGX)$ (or $H_G^*(\IGX)$).
\end{prop}
The class $\stack{S}$ is also used to define a new grading on the inertial Chow and K-theory.
\begin{df} The $\stack{S}$-age of a connected component $[U/G]$ of $\IGX$ is defined to be $\age_{\stack{S}}(U)=\rk(\stack{S}|_U).$  For an element $\stack{F}\in K_G(\IGX)$ supported on $U$, its $\stack{S}$-degree is $$\deg_{\stack{S}}\stack{F}=\age_{\stack{S}}(U)\mod\mathbb{Z},$$ giving a $\mathbb{Q}/\mathbb{Z}$-grading on $K_G(\IGX)$.  The $\stack{S}$-grading of an element $x\in A^*_G(\IGX)$ is the rational number $$\deg_{\stack{S}}x|_U=\deg x|_U+\age_{\stack{S}}(U),$$ which gives a $\mathbb{Q}$-grading on Chow. For a class in cohomology, $x\in H_G^*(\IGX)$, the $\stack{S}$-grading is the rational number $$\deg_{\stack{S}}x|_U=\deg x|_U+2\age_{\stack{S}}(U).$$  Denote by $A^{\{q\}}_G(\IGX)$ the subspace of $A^*_G(\IGX)$ of elements of $\stack{S}$-degree $q\in\mathbb{Q}^l$, where $l$ is the number of of connected components of $I\stack{X}$.  The restriction of the inertial Chern character to $A^{\{0\}}_G(\IGX)$, denoted $\inCh^0$, is the inertial rank for $\stack{S}$.
\end{df}
Proposition 3.11 in \cite{EJK1} states that the inertial Chern character homomorphism preserves the $\stack{S}$-degree modulo $\mathbb{Z}$.

With respect to the usual product on K-theory, have the following definition.
\begin{df} Let $Y$ be a manifold, and $G$ a closed algebraic group.  The \emph{augmentation homomorphism} on equivariant K-theory $\epsilon:K_G(Y)\rightarrow K_G(Y)$ is given by $$\epsilon(\stack{F}|_U)=\Ch^0(\stack{F}|_U)\stack{O}_U$$ where $[U/G]$ is a connected component of $[Y/G]$.  If $\stack{F}$ is a vector bundle on $U$, then $\Ch^0(\stack{F}|_U)$ is the rank of $\stack{F}$ on $U$.  The kernel of $\epsilon$, $\mathfrak{a}_Y$, is called the \emph{augmentation ideal}.
\end{df}
\begin{df}  The completion of the ring $K_G(Y)$ with respect to the augmentation ideal is called the \emph{augmentation completion}, $\widehat{K}_G(Y)_{\mathbb{Q}}$.
\end{df}

\begin{df}
An \emph{inertial augmentation homomorphism} is $\widetilde{\epsilon}:K_G(\IGX)\rightarrow K_G(\IGX)$, given by $$\widetilde{\epsilon}(\stack{F}|_U)=\inCh^0(\stack{F}|_U)\stack{O}_U,$$ where $U$ is as above.  We call $(K_G(\IGX),*,1,\widetilde{\epsilon})$ the \emph{augmented inertial ring}.
\end{df}
\begin{rem} By the definition of $\inCh^0$, we have that $$\widetilde{\epsilon}(\stack{F}|_U)=\begin{cases} {\epsilon}(\stack{F}|_U), \hspace{.1in}\textup{if}\,\,\age_{\stack{S}}(U)=0 \\ 0, \hspace{.1in}\textup{otherwise} \end{cases}.$$
\end{rem}

The augmented inertial ring above has compatible Adams operations, as shown in \cite{EJK2}.
\begin{df} Let $R$ be a commutative, unital ring together with a collection of ring homomorphisms $\psi^n:R\rightarrow R$ for all $n\ge 1$.  $R$ is called a $\psi$-ring if for all $x\in R$ and for all $n\ge 1$,
\begin{enumerate}
\item $\psi^1(x)=x$
\item $\psi^n(\psi^m(x))=\psi^{nm}(x)$
\end{enumerate}
The homomorphisms $\psi^n$ are called \emph{Adams} (or \emph{power}) \emph{operations}.
\end{df}
We will also need to consider $\psi$-rings with an augmentation.
\begin{df}
If $(R,\cdot,1,\epsilon)$ is an augmented ring, then $(R,\cdot,1,\psi,\epsilon)$ is an augmented $\psi$-ring if
\begin{equation} \label{eq.augpsicomp} \epsilon(\psi^k(\stack{F}))=\psi^k(\epsilon(\stack{F}))=\epsilon(\stack{F}),
\end{equation}
for all $k\ge 1$.  Set $\psi^0=\epsilon$.
\end{df}

Recall that ordinary equivariant K-theory $K_G(X)$ is a $\psi$-ring whose Adams operations $\psi^k:K_G(X)\to K_G(X)$ are defined though the $\lam$-ring structure of $K_G(X)$. That is, for all $i\geq 0$, define maps $\lam^i:K_G(X)\to K_G(X)$  by demanding that for any $G$-equivariant vector bundle $V$, $\lam^i([V]) := [\Lambda^i V]$, the class of the $i$-th exterior power of $V$, and also by demanding that the series $\lam_t := \sum_{i\geq 0} t^i \lam^i : K_G(X)\to K_G(X)[[t]]$ satisfy the multiplicativity relation $\lam_t(\cf_1+\cf_2) = \lam_t(\cf_1)\lam_t(\cf_2)$ for all $\cf_1, \cf_2$ in $K_G(X)$. These $\lam$-operations make $K_G(X)$ into a so-called $\lam$-ring. The Adams operations are defined through the equality of power series
\begin{equation}\label{eq:LamPsi}
\lam_t(\cf) = \exp\left( \sum_{n\geq 1}(-1)^{n-1} \frac{t^n}{n} \psi^n(\cf)\right).
\end{equation}
In the special case where
 $\stack{L}$ is the class of a line bundle, it follows that
\begin{equation} \label{eq.ordlineelt} \psi^k(\stack{L})=\stack{L}^{\otimes k}.
\end{equation}

If an element $\cf$ in $K_G(X)$ is the class of a rank $n$ vector bundle then $\lam_t(\cf)$ is a degree $n$ polynomial in $t$ and $\lam^n(\cf)$ is an invertible element of $K_G(X)$. An element $\cf$ in $K_G(X)$ which satisfies these properties is said to be a \textsl{$\lam$-positive element of rank $n$.} Although a rank $n$ $\lam$-positive element $\cf$ need not be the class of a rank $n$ vector bundle, in general, they do share many of the properties of vector bundles, e.g. they possess an Euler class in K-theory, Chow and cohomology.  

We will now define the inertial Adams operations. We begin by introducing Bott classes.
\begin{df} Let $Y$ be a manifold, and let $\stack{L}$ be an equivariant line bundle in $K_G(Y)$.  For each $j\ge 1$, the \emph{j-th Bott class of $\stack{L}$} is
\begin{equation} \label{eq.bott} \theta^j(\stack{L})=\frac{1-\stack{L}^j}{1-\stack{L}}=\sum_{i=0}^{j-1}\stack{L}^i.
\end{equation}
The splitting principle is used to extend this definition to any equivariant vector bundle in $K_G(Y)$.  The result is a multiplicative class, where for any vector bundles $\stack{F}_1$ and $\stack{F}_2$, $$\theta^j(\stack{F}_1+\stack{F}_2)=\theta^j(\stack{F}_1)\theta^j(\stack{F}_2).$$
\end{df}
\begin{df}
If $\stack{S}$ can be represented by a vector bundle, then the inertial pair is called \emph{strongly Gorenstein}.
\end{df}
\begin{df}
Suppose the inertial pair $(\stack{R},\stack{S})$ is strongly Gorenstein.  Then for all $k\ge 1$, the \emph{inertial Adams operations} $\widetilde{\psi}^k:K_G(\IGX)\rightarrow K_G(\IGX)$ are given by
\begin{equation} \label{eq.inpsi} \widetilde{\psi}^k(\stack{F})=\psi^k(\stack{F})\cdot\theta^k(\stack{S}^*),
\end{equation}
where $\cdot$ is the usual product on $K_G(\IGX)$.
\end{df}
We are motivated by equation \eqref{eq.ordlineelt} to introduce the following analog of classes of line bundles in ordinary equivariant K-theory.
\begin{df} An invertible element $\stack{L}$ in the inertial K-theory ring $K_G(\IGX)$ satisfying equation \eqref{eq:VirtualLineElt},
\begin{equation*} \widetilde{\psi}^k(\stack{L})=\stack{L}^{*k},
\end{equation*}
for all $k\ge 1$ is called an \emph{inertial line element}.
\end{df}
Notice that after tensoring with $\nq$, an inertial line element in $K_G(\IGX)_\nq$ is nothing more than a $\lam$-positive element of rank $1$ with respect to the inertial $\lam$-ring structure $\olam_t:K_G(\IGX)_\nq\to K_G(\IGX)_\nq$, where $\olam_t := \sum_{i\geq 0} t^i \olam^i$ is defined from the inertial Adams operations  $\psit^k$ through Equation (\ref{eq:LamPsi}) but where $\psi^n$ is replaced by $\psit^n$, $\lam_t$ is replaced by $\olam_t$, and all products are understood to be the inertial product.
\begin{thm}[\cite{EJK2},Theorem 5.16] Let $G$ be a diagonalizable group and let $(\stack{R},\stack{S})$ be a strongly Gorenstein inertial pair on $\IGX$.  The ring $(K_G(\IGX),*_{\stack{R}},\mathbf{1},\widetilde{\epsilon},\widetilde{\psi})$ is an augmented $\psi$-ring.
\end{thm}

The virtual orbifold product was defined in \cite{GLSUX}, and in \cite{EJK1} was shown to arise from a particular inertial pair.
\begin{df} The tangent bundle of $\stack{X}$ gives an equivariant bundle $\mathbb{T}\in K_G(X)$ on $X$.  Denote by $\mathbb{T}|_{\Imult{2}_GX}$ the pullback of $\mathbb{T}$ to $\Imult{2}_GX$ via the projection map from $\Imult{2}_GX$ to $X$.  Further, let $\mathbb{T}_{\IGX}$ be the tangent bundle of $I\stack{X}$ and $\mathbb{T}_{\Imult{2}_GX}$ be the tangent bundle of $I^2\stack{X}$.  The virtual obstruction class $\stack{R}=\mathbb{T}^{virt}\in K_G(\Imult{2}_GX)$ is defined by \begin{equation} \label{eq.virr} \mathbb{T}|_{\Imult{2}_GX}+\mathbb{T}_{\Imult{2}_GX}-e_1^*\mathbb{T}_{\IGX}-e_2^*\mathbb{T}_{\IGX}.\end{equation}
\end{df}
\begin{df} Let $\mathbf{N}$ be the normal bundle of the projection $\IGX\rightarrow X$
\end{df}
\begin{prop}[\cite{EJK1},Proposition 4.3.2] The pair $(\mathbb{T}^{virt},\mathbf{N})$ is a strongly Gorenstein inertial pair.
\end{prop}
\begin{crl}[\cite{EJK1},Corollary 4.3.4]  The virtual orbifold Chow ring from \cite{GLSUX} is isomorphic as a ring to the inertial Chow ring associated to the inertial pair $(\mathbb{T}^{virt},\mathbf{N})$.
\end{crl}

\begin{df}
For the virtual K-theory above, we will call $\widetilde{\psi}^k$ \emph{virtual Adams operations} and $\widetilde{\epsilon}$ the \emph{virtual augmentation}.
\end{df}

In \cite{EJK2}, a variant of the hyper-K\"{a}hler resolution conjecture was introduced for inertial K-theory.  Let $\stack{X}=[X/G]$ be an orbifold where $G$ is diagonalizable, and let $Z$ be a hyper-K\"{a}hler resolution of $\mathbb{T}^*\stack{X}$.  There is an isomorphism between $\widehat{K}(I\stack{X})_{\mathbb{C}}$ with the virtual orbifold product and $K(Z)_{\mathbb{C}}$ with the usual product.

\section{Virtual K-theory of $\pn$}

The statements and results in this section are a synopsis of the results from \S7 of \cite{EJK2}.
Let $X=\mathbb{C}^2-\{(0,0)\}$ have an action of the group $\mathbb{C}^*$ given by $\lambda \cdot (z_1,z_2)=(\lambda z_1,\lambda^n z_2)$, for a positive integer $n\ge 2$.  Then the quotient space $[X/\mathbb{C}^*]$ is the global quotient orbifold $\pn$.  As a vector space, the virtual orbifold K-theory is given by the equivariant K-theory of the inertia manifold $I_GX=\coprod_{m=0}^{n-1} X^m$, the disjoint union of the fixed point set of $X$ under the action of an element of the group $\mathbb{Z}/n\mathbb{Z}=\{0,\ldots,n-1\}$ using additive notation.  We view $\mathbb{Z}/n\mathbb{Z}$ as a subgroup of $\mathbb{C}^*$ under the identification $m\mapsto \zeta^m$, for $\zeta=exp(2\pi i/n)$.   With respect to the usual product, when $m=1$, this gives the ring $$K_G(X^0)\cong \frac{\mathbb{Z}[x_0^{\pm1}]}{\langle(x_0-1)(x_0-1)^n\rangle},$$ and when $m\in\{1,\ldots,n-1\}$, $$K_G(X^m)\cong \frac{\mathbb{Z}[x_m^{\pm1}]}{\langle x_m^n-1\rangle}.$$  Denote the identities in each ring above by $1_m$ where $m\in\{0,\ldots,n-1\}$.  Equation \eqref{eq.virr} gives the formula for $\stack{R}$ and $\stack{S}_m:=\stack{S}|_{X^m}=x_m$, so the virtual multiplication in this case can be written as 
\begin{equation} \label{eq.virprodp1n} x_{m_1}^{a_1}*x_{m_2}^{a_2}=x_{m_1+m_2}^{a_1+a_2}\cdot\euler(\stack{S}_{m_1}+\stack{S}_{m_2}-\stack{S}_{m_1+m_2}),
\end{equation}
where the Euler class is
\[
\euler(\stack{S}_{m_1}+\stack{S}_{m_2}-\stack{S}_{m_1 + m_2}) = 
\begin{cases}
1 &\mbox{if either $m_1=0$ or $m_2 = 0$,}\\
1_{m_1+m_2}-2x_{m_1 + m_2}^{-1}+x_{m_1+m_2}^{-2}&\mbox{if $m_1+m_2 = n$, 
$m_1\not=0$ and $m_2\not=0$,} \\
1_{m_1+m_2}-x_{m_1 + m_2}^{-1}&\mbox{otherwise}.
\end{cases}
\]
The virtual augmentation is given by $\widetilde{\epsilon}(x_0^a)=1_0$ and $\widetilde{\epsilon}(x_m^a)=0$ for $m\in\{1,\ldots,n-1\}$.  For the definition of the virtual Adams operations, see equation \eqref{eq.inpsi}.  Here, the Bott classes are given by $$\theta^j(\stack{S}_m^*)=\theta^j(x_m^{-1})=\sum_{i=0}^{j-1}x_m^{-i}.$$

In Proposition 7.22 of \cite{EJK2}, Edidin, Jarvis, and Kimura  prove that a subring of the inertial K-theory of $\mathbb{P}(1,2)$ and $\mathbb{P}(1,3)$ is isomorphic to the ordinary K-theory of a toric crepant resolution of the cotangent bundle of $\mathbb{P}(1,2)$ and $\mathbb{P}(1,3)$, respectively.  The cotangent bundle to $\pn$ is the quotient stack $[(X\times\mathbb{A}^1)/\mathbb{C}^*]$ where the action of $\mathbb{C}^*$ has weights $(1,n,-(n+1))$.  There is a simplicial fan associated to this quotient, and a subdivision of this fan determines the toric crepant resolution of the cotangent bundle.  See \cite{EJK2} \S7.3 for more details.  This resolution has ordinary K-theory
\begin{equation} \label{eq.kres} K(Z_n)=\frac{\mathbb{Z}[\chi_0^{\pm1},\ldots,\chi_{n-1}^{\pm1}]}{\langle\euler(\chi_0),\ldots,\euler(\chi_{n-1})\rangle^2}=\frac{\mathbb{Z}[\chi_0^{\pm1},\ldots,\chi_{n-1}^{\pm1}]}{\langle(\chi_0-1),\ldots,(\chi_{n-1}-1)\rangle^2}.
\end{equation}

\section{Localization}

Henceforth, we will assume complex coefficients unless otherwise stated.  Edidin and Graham \cite{EG} proved the following proposition.
\begin{prop}  Let $G$ be an algebraic group that acts on an algebraic space $Y$ with finite stabilizers.  Then, the localization maps give a direct sum decomposition $$\Gamma:K_G(Y)\stackrel{\sim}{\longrightarrow}\bigoplus_{i}K_G(Y)_{\mathfrak{m}_{\Phi_i}}$$
\end{prop}
In the case of $\pn$, we take $Y=\IGX$, and the maximal ideals $\mathfrak{m}_i\subseteq R(G)\cong \mathbb{C}[x^{\pm 1}]$, for $i\in\mathbb{Z}/n\mathbb{Z}$, are the maximal ideals of characters vanishing on $\zeta^i$.  Thus, $\mathfrak{m}_0=\langle x-1\rangle$ and $\mathfrak{m}_i=\langle x-\zeta^i\rangle$ where $\zeta=exp(2\pi i/n)$.  We define $$\mathcal{K}=\bigoplus_{l=0}^{n-1}\mathcal{K}_i,$$ where the summands $\mathcal{K}_l$ are the localizations above $$\mathcal{K}_l=K_G(IX)_{\mathfrak{m}_l}.$$  Define generators of the direct sum decomposition of localizations by $$x_{ml}:=\Gamma(x_m)|_{K_G(X^m)_{\mathfrak{m}_l}}.$$  The result is a decomposition as a vector space of the K-theory of $\pn$, as in the diagram below.

$$\scalebox{1.35}{\xymatrix@=0pt{\frac{\mathbb{C}[x_0^{\pm1}]}{\langle(x_0-1_0)(x_0^n-1_0)\rangle} \ar[ddddd]_{\Gamma} &  \oplus & \frac{\mathbb{C}[x_1^{\pm1}]}{\langle x_1^n-1_1\rangle} \ar[ddddd]_{\Gamma} & \oplus &\cdots &\oplus &  \frac{\mathbb{C}[x_{n-1}^{\pm1}]}{\langle x_{n-1}^n-1_{n-1}\rangle}\ar[ddddd]_{\Gamma} \\ &&&&&& \\ &&&&&& \\ &&&&&& \\ &&&&&& \\  \frac{\mathbb{C}[x_{00}^{\pm1}]}{I_{00}} & \oplus & \frac{\mathbb{C}[x_{10}^{\pm1}]}{I_{10}} & \oplus & \cdots & \oplus & \frac{\mathbb{C}[x_{(n-1)0}^{\pm1}]}{I_{(n-1)0}} \\ \oplus && \oplus &&  && \oplus\\ \frac{\mathbb{C}[x_{01}^{\pm1}]}{I_{01}} & \oplus & \frac{\mathbb{C}[x_{11}^{\pm1}]}{I_{11}} & \oplus & \cdots & \oplus & \frac{\mathbb{C}[x_{(n-1)1}^{\pm1}]}{I_{(n-1)1}} \\ \\ \oplus && \oplus &&  && \oplus\\ \vdots && \vdots &&  && \vdots \\ \\ \oplus && \oplus &&  && \oplus\\ \frac{\mathbb{C}[x_{0(n-1)}^{\pm1}]}{I_{0(n-1)}} & \oplus & \frac{\mathbb{C}[x_{1(n-1)}^{\pm1}]}{I_{1(n-1)}} & \oplus & \cdots & \oplus & \frac{\mathbb{C}[x_{(n-1)(n-1)}^{\pm1}]}{I_{(n-1)(n-1)}}}}$$
where $I_{ml}=\langle x_{ml}-\zeta^l1_{ml}\rangle$ for all $m$ and $l$, except for $I_{00}=\langle(x_{00}-1_{00})^2\rangle$.

\begin{prop} The rows in $\mathcal{K}$ are the localizations $\mathcal{K}_l$ for $l=0,\ldots,n-1$, and the inverse images of the generators of $\mathcal{K}$ are 
\begin{align} 
\Gamma^{-1}(1_{00}) &= \frac{1}{2n}((1-n)x_0+(1+n))\frac{x_0^n-1}{x_0-1} \label{eq.100} \\ 
\Gamma^{-1}(x_{00}) &= \frac{1}{2n}((3-n)x_0+(n-1))\frac{x_0^n-1}{x_0-1}  \label{eq.x00} \\ 
\Gamma^{-1}(1_{0l}) &= \frac{\zeta^l}{n(\zeta^l-1)}(x_0-1)\frac{x_0^n-1}{x_0-\zeta^l} \hspace{.2in}\textup{for}\,\,l\neq0 \label{eq.10l} \\ 
\Gamma^{-1}(1_{ml}) &= \frac{\zeta^l}{n}\frac{x_m^n-1}{x_m-\zeta^l} \hspace{.2in}\textup{for}\,\,m\neq0, \label{eq.1ml} 
\end{align} 
where all products above are with respect to the ordinary product on $K_G(I\pn)_\mathbb{C}$ with complex coefficients, and where we regard $\frac{u^n-1}{u-\zeta^l}=\prod_{i\neq l}(u-\zeta^i)$.
\end{prop}

\begin{proof} We now calculate the map $\Gamma^{-1}$ on the generators of the decomposition.  Consider the element $1_{00}$.  If $\Gamma(f)=1_{00}$, then $f(x_0)=p(x_0)(x_0-1)(x_0^n-1)+q(x_0)$, where $p$ and $q$ are polynomials, and the degree of $q$ is at most $n$.  The localization maps give that $f\equiv 0\bmod{(x_0-\zeta^l)}$, where $l\neq 0$.  Thus, $q(x_0)=(ax_0+b)\prod_{l\neq 0}(x_0-\zeta^l)=(ax_0+b)\frac{x_0^n-1}{x_0-1}$, where $a,b\in\mathbb{C}$.  We also have that $f\equiv 1\bmod{(x_0-1)^2}$, or equivalently, $f(1)=1$ and $f'(1)=0$.  Hence, we have 
\begin{align*} f(1) &= \lim_{x_0\to 1}(ax_0+b)\frac{x_0^n-1}{x_0-1} = n(a+b) \\ \frac{1}{n} &= a+b. & \end{align*}  
Similarly,
\begin{align*} f'(1) &= \lim_{x_0\to 1} \frac{(n-1)x_0^n-nx_0^{n-1}+1}{(x_0-1)^2}(ax_0+b)+a\frac{x_0^n-1}{x_0-1} \\ 0 &= \frac{n(n-1)}{2}(a+b)+na = \frac{n(n+1)}{2}a+\frac{n(n-1)}{2}b. \end{align*}
Combining these two expressions, we obtain $$a=\frac{1-n}{2n} \hspace{.2in} b=\frac{1+n}{2n}.$$
The inverse of $1_{00}$ can then be written as $$\Gamma^{-1}(1_{00})=\frac{1}{2n}((1-n)x_0+(1+n))\frac{x_0^n-1}{x_0-1}.$$

The computation for $x_{00}$ is similar to that for $1_{00}$.  The only difference is now we require that $f'(1)=1$ instead of 0.  This gives $$1=\frac{n(n+1)}{2}a+\frac{n(n-1)}{2}b.$$ Solving for $a$ and $b$ now gives $$a=\frac{3-n}{2n} \hspace{.2in} b=\frac{n-1}{2n}.$$
Therefore, we obtain an expression for the inverse of $x_{00}$ given by $$\Gamma^{-1}(x_{00})=\frac{1}{2n}((3-n)x_0+(n-1))\frac{x_0^n-1}{x_0-1}.$$ 

Now consider the element $1_{0l}$, where $l\neq 0$.  As before, we assume $1_{0l}=\Gamma(f(x_0))=p(x_0)(x_0-1)(x_0^n-1)+g(x_0)$ where the degree of $g$ is at most $n$.  As $1_{0l}\in\frac{\mathbb{C}[x_{0l}^{\pm1}]}{\langle x_{0l}-\zeta^l1_{0l}\rangle}$, we have $f\equiv 0\,mod(x-\zeta^k)$ for $k\neq l$.  This allows us to write $$g(x_0)=r(x_0)\prod_{k\neq l}(x_0-\zeta^k)=r(x_0)\frac{x_0^n-1}{x_0-\zeta^l},$$ for a polynomial $r$ of degree at most one.  Note that $f\equiv 0\bmod{(x_0-1)^2}$, so $r(x_0)=\alpha(x_0-1)$ for some constant $\alpha$.
We also have $f\equiv 1\bmod{(x-\zeta^l)}$.  So we have
$$f(\zeta^l) = \lim_{x_0\to\zeta^l}\alpha(x_0-1)\frac{x_0^n-1}{x_0-\zeta^l} = \lim_{x_0\to\zeta^l}\alpha(x_0-1)\frac{nx_0^{n-1}}{1} = \alpha(\zeta^l-1)n\zeta^{-l}.$$
And therefore, $$\alpha = \frac{\zeta^l}{n(\zeta^l-1)}.$$

We now have a description of the inverse of $1_{0l}$ $$\Gamma^{-1}(1_{0l})=\frac{\zeta^l}{n(\zeta^l-1)}(x_0-1)\frac{x_0^n-1}{x_0-\zeta^l}.$$

The last generators are $1_{ml}$ where $m$ is nonzero.  In this case, $1_{ml}=\Gamma(f(x_m))=p(x_m)(x_m-1)+q(x_m)$ where the degree of $q$ is at most $n-1$.  Since $1_{ml}\in\frac{\mathbb{C}[x_{ml}^{\pm1}]}{\langle x_{ml}-\zeta^l1_{ml}\rangle}$, we have $f\equiv 0\bmod{(x-\zeta^k)}$ for $k\neq l$.  Therefore, $$g(x_m)=\alpha\prod_{k\neq l}(x_m-\zeta^k)=\alpha\frac{x_m^n-1}{x_m-\zeta^l},$$ where $\alpha$ is a constant.  Furthermore, $f\equiv 1\bmod{(x-\zeta^l)}$, so we can solve for $\alpha$,
$$f(\zeta^l) = \lim_{x_m\to\zeta^l}\alpha\frac{x_m^n-1}{x_m-\zeta^l} = \lim_{x_m\to\zeta^l}\alpha\frac{nx_m^{n-1}}{1} = \alpha n\zeta^{-l}.$$
We then get $$\alpha = \frac{\zeta^l}{n}.$$
The inverse of $1_{ml}$ is given by $$\Gamma^{-1}(1_{ml})=\frac{\zeta^l}{n}\cdot\frac{x_m^n-1}{x_m-\zeta^l}.$$
\end{proof}

\section{The virtual product on the localization}
In the previous section, we constructed an inverse to the localization maps.  Using the vector space isomorphism $\Gamma$, and the virtual product on $\pn$ as given in equation \eqref{eq.virprodp1n}, a virtual product can be constructed on the localization $\mathcal{K}=\bigoplus_{l=0}^{n-1}\mathcal{K}_l$.

\begin{prop} The virtual product on $\mathcal{K}$ is given by linearly extending the virtual multiplications on the generators of $\mathcal{K}$ below.
\begin{align*} 1_{00}*1_{00} &= 1_{00} & 1_{00}*x_{00} &= x_{00} \\ 1_{00}*1_{m0} &= 1_{m0}, \hspace{.1in}\textup{if $m\neq 0$} & x_{00}*x_{00} &= 2x_{00}-1_{00} \\ x_{00}*1_{m0} &= 1_{m0}, \hspace{.1in} \textup{if $m\neq 0$} & 1_{m_10}*1_{m_20} &= 0 \\ 1_{0l}*1_{ml} &= 1_{ml} & x_{00}*1_{ml} &= 0, \hspace{.1in} \textup{if $l\neq 0$} \\ 1_{m_1l}*1_{m_2l} &= (1-\zeta^{-l})1_{(m_1+m_2)l}, \hspace{.1in} \textup{if $m_1+m_2\neq n,\,\, m_1,m_2\neq 0$} \\ 1_{m_1l}*1_{m_2l} &= (1-\zeta^{-l})^21_{0l}, \hspace{.1in} \textup{if $m_1+m_2=n,\,\, m_1,m_2\neq 0$} \\ 1_{m_1l_1}*1_{m_2l_2} &= 0, \hspace{.1in} \textup{if $l_1\neq l_2$} &  \end{align*}
\end{prop}

\begin{proof}
We use the definition of virtual multiplication in equation \eqref{eq.virprodp1n} to multiply the inverse of elements in $\mathcal{K}$, and then apply $\Gamma$ to obtain a multiplication in $\mathcal{K}$.
\begin{align*} 
\Gamma^{-1}(1_{00})*\Gamma^{-1}(1_{00}) &= \frac{1}{4n^2}((1-n)x_0^n+2x_0^{n-1}+\cdots+2x_0+(1+n))*((1-n)x_0^n+2x_0^{n-1} \\ & \quad +\cdots+2x_0+(1+n)) \\
&= \frac{1}{4n^2}((1-2n+n^2)x_0^{2n}+\sum_{k=1}^{n-1}(4(1-n)+4(k-1))x_0^{2n-k} \\ & \quad +(2(1-n^2)+4(n-1))x_0^n+\sum_{k=1}^{n-1}(4(1+n)+4(n-k-1))x_0^{n-k}+(1+2n+n^2)) \\ 
&= \frac{1}{4n^2}((1-2n+n^2)x_0^{2n}+\sum_{k=1}^{n-1}(4k-4n)(x_0^n+x_0^{n-k}-1) \\ & \quad +(2(1-n^2)+4(n-1))x_0^n+\sum_{k=1}^{n-1}(8n-4k)x_0^{n-k}+(1+2n+n^2)) \\ 
 &= \frac{1}{4n^2}((-2n^2+2n)x_0^n+\sum_{k=1}^{n-1}4nx_0^{n-k}+(2n^2+2n)) \\
 &= \frac{1}{2n}((1-n)x_0^n+\sum_{k=1}^{n-1}2x_0^{n-k}+(1+n)) \\ 
 &= \Gamma^{-1}(1_{00}) .
 \end{align*}
Taking $\Gamma$ of both sides of the equation above gives the multiplication $$1_{00}*1_{00}=1_{00}.$$
\begin{align*}
\Gamma^{-1}(1_{00})*\Gamma^{-1}(x_{00}) &= \frac{1}{4n^2}((1-n)x_0^{n}+2x_0^{n-1}+\cdots+2x_0+(1+n))*((3-n)x_0^{n}+2x_0^{n-1}\\ & \quad +\cdots+2x_0+(n-1)) \\
&= \frac{1}{4n^2}((3-4n+n^2)x_0^{2n}+\sum_{k=1}^{n-1}(4-4n+4k)x_0^{2n-k}+(-2n^2+8n-2)x_0^n\\ & \quad +\sum_{k=1}^{n-1}(8n-4-4k)x_0^{n-k}+(n^2-1)) \\
&= \frac{1}{4n^2}((3-4n+n^2)x_0^{2n}+\sum_{k=1}^{n-1}(4-4n+4k)(x_0^n+x_0^{n-k}-1)+(-2n^2+8n-2)x_0^n\\ & \quad +\sum_{k=1}^{n-1}(8n-4-4k)x_0^{n-k}+(n^2-1)) \\
&= \frac{1}{4n^2}(((6-8n+2n^2)+\sum_{k=1}^{n-1}(4-4n-4k)+(-2n^2+8n-2))x_0^n\\ & \quad +\sum_{k=1}^{n-1}(4n)x_0^{n-k}+(-(3-4n+n^2)-\sum_{k=1}^{n-1}(4-4n-4k)+(n^2-1))) \\
&= \frac{1}{4n^2}((-2n^2+6n)x_0^n+\sum_{k=1}^{n-1}(4n)x_0^{n-k}+(2n^2-2n)) \\
&= \frac{1}{2n}((3-n)x_0^n+2x_0^{n-1}+\cdots+2x_0+(n-1)) \\
&= \Gamma^{-1}(x_{00}).
\end{align*}
This gives the multiplication $$1_{00}*x_{00}=x_{00}.$$
\begin{align*}
\Gamma^{-1}(1_{00})*\Gamma^{-1}(1_{m0}) &= \frac{1}{2n^2}((1-n)x_0^{n}+2x_0^{n-1}+\cdots+2x_0+(1+n))*(x_m^{n-1}+\cdots+x_m+1) \\ 
&= \frac{1}{2n^2}(\sum_{k=1}^n((1-n)+(2(k-1)))x_m^{2n-k}+\sum_{k=1}^n(2(n-k)+(1+n))x_m^{n-k}) \\
&= \frac{1}{2n^2}(\sum_{k=1}^n(2k-n-1)x_m^{n-k}+\sum_{k=1}^n(3n-2k+1)x_m^{n-k}) \\
&= \frac{1}{2n^2}(\sum_{k=1}^n2nx_m^{n-k}) \\
&= \frac{1}{n}(x_m^{n-1}+\cdots+x_m+1) \\
&= \Gamma^{-1}(1_{m0}).
\end{align*}
Thus, we have $$1_{00}*1_{m0}=1_{m0}.$$
\begin{align*}
\Gamma^{-1}(x_{00})*\Gamma^{-1}(x_{00}) &= \frac{1}{4n^2}((3-n)x_0^{n}+2x_0^{n-1}+\cdots+2x_0+(n-1))*((3-n)x_0^{n}+2x_0^{n-1}\\ & \quad +\cdots+2x_0+(n-1)) \\ 
&= \frac{1}{4n^2}((9-6n+n^2)x_0^{2n}+\sum_{k=1}^{n-1}(4(3-n)+4(k+1))x_0^{2n-k}\\ & \quad +(2(-n^2+4n-3)+4(n-1))x_0^n+\sum_{k=1}^{n-1}(4(n-1)+4(n-1-k))x_0^{n-k}\\ & \quad +(n^2-2n+1)) \\
&= \frac{1}{4n^2}((9-6n+n^2)(2x_0^n-1)+\sum_{k=1}^{n-1}(4(k+2)-4n)(x_0^n+x_0^{n-k}-1)\\ & \quad +(-2n^2+12n-10)x_0^n+\sum_{k=1}^{n-1}(8n-8-4k)x_0^{n-k}+(n^2-2n+1)) \\
&= \frac{1}{4n^2}((-2n^2+10n)x_0^n+\sum_{k=1}^{n-1}4nx_0^{n-k}+(2n^2-6n))\\
&= \frac{1}{2n}((5-n)x_0^n+2x_0^{n-1}+\cdots+2x_0+(n-3)) \\
&= 2\Gamma^{-1}(x_{00})-\Gamma^{-1}(1_{00}).
\end{align*}
Therefore, $$x_{00}*x_{00}=2x_{00}-1_{00}.$$ and more generally, $$x_{00}^k=kx_{00}-(k-1)1_{00}.$$
\begin{align*} 
\Gamma^{-1}(x_{00})*\Gamma^{-1}(1_{m0}) &= \frac{1}{2n^2}((3-n)x_0^{n}+2x_0^{n-1}+\cdots+2x_0+(n-1))*(x_m^{n-1}+\cdots+x_m+1) \\
&= \frac{1}{2n^2}(\sum_{k=1}^n((3-n)+2(k-1))x_m^{2n-k}+\sum_{k=1}^n(2(n-k)+(n-1))x_m^{n-k}) \\
&= \frac{1}{2n^2}(\sum_{k=1}^n(2k+1-n)x_m^{n-k}+\sum_{k=1}^n(3n-2k-1)x_m^{n-k}) \\
&= \frac{1}{2n^2}(\sum_{k=1}^n2nx_m^{n-k}) \\ 
&= \frac{1}{n}(x_m^{n-1}+\cdots+x_m+1) \\
&= \Gamma^{-1}(1_{m0}).
\end{align*}
This calculation shows that $$x_{00}*1_{m0}=1_{m0}.$$
Suppose that $m_1+m_2\neq n$.
\begin{align*} 
\Gamma^{-1}(1_{m_10})*\Gamma^{-1}(1_{m_20}) &= \frac{1}{n^2}(x_{m_1}^{n-1}+\cdots+x_{m_1}+1)*(x_{m_2}^{n-1}+\cdots+x_{m_2}+1) \\
&= \frac{1}{n^2}(\sum_{k=1}^nk\xmm^{2n-k-1}(1-\xmm^{n-1})+\sum_{k=1}^{n-1}(n-k)\xmm^{n-k-1}(1-\xmm^{n-1}) \\
&= \frac{1}{n^2}(\sum_{k=1}^nk(\xmm^{2n-k-1}-\xmm^{3n-k-2})+\sum_{k=1}^{n-1}(n-k)(\xmm^{n-k-1}-\xmm^{2n-k-2})) \\
&= \frac{1}{n^2}(\sum_{k=1}^nk(\xmm^{n-k-1}-\xmm^{n-k-2})+\sum_{k=1}^{n-1}(n-k)(\xmm^{n-k-1}-\xmm^{n-k-2})) \\
&= \frac{1}{n^2}(\sum_{k=1}^nn(\xmm^{n-k}-\xmm^{n-k-1})) \\
&= 0.
\end{align*}
This identity gives $$1_{m_10}*1_{m_20}=0, \hspace{.1in}\textup{if}\,\, m_1+m_2\neq n.$$  Now suppose that $m_1+m_2=n$, then
\begin{align*} 
\Gamma^{-1}(1_{m_10})*\Gamma^{-1}(1_{m_20}) &= \frac{1}{n^2}(x_{m_1}^{n-1}+\cdots+x_{m_1}+1)*(x_{m_2}^{n-1}+\cdots+x_{m_2}+1) \\
&= \frac{1}{n^2}(\sum_{k=1}^nkx_0^{2n-k-1}(x_0^{n-2}-2x_0^{n-1}+x_0^n)+\sum_{k=1}^{n-1}(n-k)x_0^{n-k-1}(x_0^{n-2}-2x_0^{n-1}+x_0^n)) \\
&= \frac{1}{n^2}(\sum_{k=1}^nk(x_0^{3n-k-3}-2x_0^{3n-k-2}+x_0^{3n-k-1})+\sum_{k=1}^{n-1}(n-k)(x_0^{2n-k-3}-2x_0^{2n-k-2}+x_0^{2n-k-1})) \\
&= \frac{1}{n^2}(\sum_{k=1}^nk(x_0^{n-k-3}-2x_0^{n-k-2}+x_0^{n-k-1})+\sum_{k=1}^{n-1}(n-k)(x_0^{n-k-3}-2x_0^{n-k-2}+x_0^{n-k-1})) \\
&= \frac{1}{n}(\sum_{k=1}^n(x_0^{n-k-3}-2x_0^{n-k-2}+x_0^{n-k-1})) \\ 
&= 0.
\end{align*}
Here, the fourth equality uses the relations $x_0^{2n-k}=x_0^n+x_0^{n-k}-1$ and $x_0^{3n-k}=2x_0^{n}+x_0^{n-k}-2$ for $0\le k\le n$.  The conclusion is that the elements $1_{m0}$ are nilpotent, i.e. $$1_{m_10}*1_{m_20}=0, \hspace{.1in}\textup{if}\,\, m_1+m_2=n.$$  This concludes the calculation of the virtual product on $\mathcal{K}_0$.

We will calculate the virtual product on $\mathcal{K}_l$ similarly.  Henceforth, assume that $l\neq 0$.
First, suppose that $m_1+m_2\neq n$.
\begin{align*}
\Gamma^{-1}(1_{m_1l})*\Gamma^{-1}(1_{m_2l}) &= \frac{1}{n^2}(\zeta^{-l(n-1)}x_{m_1}^{n-1}+\cdots+\zeta^{-l}x_{m_1}+1)*(\zeta^{-l(n-1)}x_{m_2}^{n-1}+\cdots+\zeta^{-l}x_{m_2}+1) \\
&= \frac{1}{n^2}(\sum_{k=1}^nk\zeta^{-l(2n-k-1)}\xmm^{2n-k-1}(1-\xmm^{n-1})+\sum_{k=1}^{n-1}(n-k)\zeta^{-l(n-k-1)}\xmm^{n-k-1}(1-\xmm^{n-1})) \\
&= \frac{1}{n^2}(\sum_{k=1}^nk\zeta^{-l(2n-k-1)}(\xmm^{2n-k-1}-\xmm^{3n-k-2})+\sum_{k=1}^{n-1}(n-k)\zeta^{-l(n-k-1)}(\xmm^{n-k-1}-\xmm^{2n-k-2})) \\
&= \frac{1}{n^2}(\sum_{k=1}^nk\zeta^{-l(n-k-1)}(\xmm^{n-k-1}-\xmm^{n-k-2})+\sum_{k=1}^{n-1}(n-k)(\zeta^{-l(n-k-1)}(\xmm^{n-k-1}-\xmm^{n-k-2})) \\
&= \frac{1}{n}\sum_{k=1}^n\zeta^{-l(n-k-1)}(\xmm^{n-k-1}-\xmm^{n-k-2}) \\
&= \frac{1}{n}\sum_{k=1}^n(\zeta^{-l(n-k-1)}-\zeta^{-l(n-k-2)})\xmm^{n-k-1} \\
&= \frac{1}{n}(1-\zeta^{-l})\sum_{k=1}^n\zeta^{-l(n-k-1)}\xmm^{n-k-1} \\
&= (1-\zeta^{-l})\Gamma^{-1}(1_{m_1+m_20}).
\end{align*}
Thus, we obtain $$1_{m_1l}*1_{m_2l}=(1-\zeta^{-l})1_{m_1+m_2l}.$$  Now suppose that $m_1+m_2=n$, then
\begin{align*}
\Gamma^{-1}(1_{m_1l})*\Gamma^{-1}(1_{m_2l}) &= \frac{1}{n^2}(\zeta^{-l(n-1)}x_{m_1}^{n-1}+\cdots+\zeta^{-l}x_{m_1}+1)*(\zeta^{-l(n-1)}x_{m_2}^{n-1}+\cdots+\zeta^{-l}x_{m_2}+1) \\
&= \frac{1}{n^2}(\sum_{k=1}^nk\zeta^{-l(2n-k-1)}x_0^{2n-k-1}(x_0^{n-2}-2x_0^{n-1}+x_0^n)+\\ & \quad +\sum_{k=1}^{n-1}(n-k)\zeta^{-l(n-k-1)}x_0^{n-k-1}(x_0^{n-2}-2x_0^{n-1}+x_0^n)) \\
&= \frac{1}{n^2}(\sum_{k=1}^nk\zeta^{-l(2n-k-1)}(x_0^{3n-k-3}-2x_0^{3n-k-2}+x_0^{3n-k-1})+\\ & \quad +\sum_{k=1}^{n-1}(n-k)\zeta^{-l(n-k-1)}(x_0^{2n-k-3}-2x_0^{2n-k-2}+x_0^{2n-k-1})) \\
&= \frac{1}{n^2}(\sum_{k=1}^nk\zeta^{-l(n-k-1)}(x_0^{n-k-3}-2x_0^{n-k-2}+x_0^{n-k-1})+\\ & \quad +\sum_{k=1}^{n-1}(n-k)\zeta^{-l(n-k-1)}(x_0^{n-k-3}-2x_0^{n-k-2}+x_0^{n-k-1})) \\
&= \frac{1}{n}\sum_{k=1}^n\zeta^{-l(n-k-1)}(x_0^{n-k-3}-2x_0^{n-k-2}+x_0^{n-k-1}) \\
&= \frac{1}{n}\sum_{k=1}^n(\zeta^{-l(n-k-3)}-2\zeta^{-l(n-k-2)}+\zeta^{-l(n-k-1)})x_0^{n-k-1} \\
&= \frac{1}{n}(1-\zeta^{-l})^2\sum_{k=1}^n\zeta^{-l(n-k-1)}x_0^{n-k-1} \\
&= (1-\zeta^{-l})^2\Gamma^{-1}(1_{0l}).
\end{align*}
Taking $\Gamma$ of both sides of the equation gives $$1_{m_1l}*1_{m_2l}=(1-\zeta^{-l})^21_{0l}.$$
Suppose that $m\neq 0$, then
\begin{align*}
\Gamma^{-1}(1_{0l})*\Gamma^{-1}(1_{ml}) &= \frac{1}{n(\zeta^l-1)}(\zeta^{-l(n-1)}x_0^n+(\zeta^{-l(n-2)}-\zeta^{-l(n-1)})x_0^{n-1}\\ & \quad +\cdots+(1-\zeta^{-l})-1)*(\zeta^{-l(n-1)}x_{m}^{n-1}+\cdots+\zeta^{-l}x_{m}+1) \\
&= \frac{1}{n^2(\zeta^l-1)}(\sum_{k=1}^n(k\zeta^{-l(n-k-1)}-(k-1)\zeta^{-l(n-k)})x_m^{2n-k}+\\ & \quad +\sum_{k=1}^n((n-k)\zeta^{-l(n-k-1)}-(n-k+1)\zeta^{-l(n-k)})x_m^{n-k}) \\
&= \frac{1}{n(\zeta^l-1)}\sum_{k=1}^n(\zeta^{-l(n-k-1)}-\zeta^{-l(n-k)})x_m^{n-k} \\
&= \frac{1}{n}\sum_{k=1}^n\zeta^{-l(n-k)}x_m^{n-k} \\
&= \Gamma^{-1}(1_{ml}).
\end{align*}
This calculation shows that $$1_{0l}*1_{ml}=1_{ml}.$$  Combining the previous three identities and using associativity of virtual multiplication, we also have $$1_{0l}*1_{0l}=1_{0l}.$$
Suppose $l_1\neq l_2$, let $\eta=l_1-l_2$, and without loss of generality, take $m_1=1=m_2$.  Consider the product
\begin{align*} \frac{x_{1}^n-1}{x_{1}-\zeta^{l_1}}*\frac{x_{2}^n-1}{x_{2}-\zeta^{l_2}} &= (\zeta^{-l_1(n-1)}x_{1}^{n-1}+\cdots+1)*(\zeta^{-l_2(n-1)}x_{2}^{n-1}+\cdots+1) \\ 
&= \sum_{k=2}^{n}\Big(\sum_{i=1}^{k-1}\zeta^{-l_1(n-i)-l_2(n-k+i)}\Big)x_2^{2n-k}(1-x_2^{n-1})\\ & \quad +\sum_{k=1}^{n}\Big(\sum_{i=0}^{n-k}\zeta^{-l_1(n-k-i)-l_2i}\Big)x_2^{n-k}(1-x_2^{n-1}) \\
&= \sum_{k=2}^{n}\Big(\sum_{i=1}^{k-1}\zeta^{-l_1(n-i)-l_2(n-k+i)}\Big)(x_2^{n-k}-x_2^{n-k-1})\\ & \quad +\sum_{k=1}^{n}\Big(\sum_{i=0}^{n-k}\zeta^{-l_1(n-k-i)-l_2i}\Big)(x_2^{n-k}-x_2^{n-k-1}) \\
&= \sum_{k=2}^{n}\Big(\sum_{i=1}^{k-1}\zeta^{i\eta+kl_2}\Big)(x_2^{n-k}-x_2^{n-k-1})\\ & \quad +\sum_{k=1}^{n}\Big(\sum_{i=k}^{n}\zeta^{i\eta+kl_2}\Big)(x_2^{n-k}-x_2^{n-k-1}) \\
&= 0,
\end{align*}
using the identity $1+\zeta^\eta+\cdots+\zeta^{(n-1)\eta}=0$, which holds since $\eta\neq 0$.  As all of the inverse images of the generators of $\mathcal{K}$ contain a factor of the form $\frac{x_{m}^n-1}{x_{m}-\zeta^{l}}$ for some $m$ and some $l$, we see that whenever $l_1\neq l_2$, $$1_{m_1l_1}*1_{m_2l_2}=0.$$
\end{proof}

\section{The virtual Adams operations on the localization}
The virtual K-theory ring has extra structure, given by the virtual Adams (or $\psi$-) operations.  As with the virtual product, we can use the isomorphism $\Gamma$ to induce virtual Adams operations on the direct sum decomposition.  Let $d=\gcd(k,n)$ for all integers $k\ge 1$.

\begin{prop} Given $k\ge 1$ and $l=0,\ldots,n-1$, let the $d$ solutions to the equivalence $ky\equiv l\bmod{n}$ be given by $y=s_i$, where $i=1,\ldots,d$, so that when $l=0$, $s_1=0$.  The virtual Adams operations $\widetilde{\psi}^k:\mathcal{K}\rightarrow\mathcal{K}$ are given by
\begin{align*} \widetilde{\psi}^k(1_{0l}) &= \sum_{i=1}^d1_{0s_i}, \hspace{.1in}\textup{if}\,\, d\,|\,l,\,l\neq0 & \widetilde{\psi}^k(1_{0l}) &= 0, \hspace{.1in}\textup{if}\,\,  d\nmid l,\,l\neq0 \\ \widetilde{\psi}^k(1_{ml}) &= \sum_{i=1}^d\frac{\zeta^{-l}-1}{\zeta^{-s_i}-1}1_{ms_i}, \hspace{.1in}\textup{if}\,\,  d\,|\,l,\,m\neq0\neq l & \widetilde{\psi}^k(1_{ml}) &= 0, \hspace{.1in}\textup{if}\,\,  d\nmid l,\,m\neq0\neq l \\ \widetilde{\psi}^k(1_{m0}) &= k1_{m0}, \hspace{.1in}\textup{if}\,\,  m\neq0 & \widetilde{\psi}^k(1_{00}) &= \sum_{i=1}^d1_{0s_i} \\ \widetilde{\psi}^k(x_{00}) &= kx_{00}-(k-1)1_{00}+\sum_{i=2}^d1_{0s_i}. & \end{align*}
\end{prop}

\begin{proof}
Suppose $l\neq 0$ and $m\neq 0$, and recall the expression for the inverse of $1_{ml}$ in equation \eqref{eq.1ml}.  Let
\begin{align*}
h(x_m) &= \tpsi^k(\Gamma^{-1}(1_{ml})) \\ 
&= \psi^k(\Gamma^{-1}(1_{ml}))\theta^k(x_m^{-1}) \\
&= \frac{\zeta^l}{n}\cdot\frac{x_m^{kn}-1}{x_m^k-\zeta^l}\cdot(1+x_m^{-1}+\cdots+x_m^{-k+1}) \\
&= f(x_m)(x_m^n-1)+g(x_m).
\end{align*}
where the degree of $g$ is at most $n-1$.  Suppose that $d$ does not divide $l$, then there are no solutions to the equivalence $ky\equiv l\bmod{n}$.  Therefore, $h(\zeta^y)=0$ for all $0\le y\le n-1$.  But the degree of $g$ is at most $n-1$, so we must have $g(x_m)=0$.  In this case, we have $$\tpsi^k(1_{ml})=0.$$
Now consider the case where $d$ divides $l$.  If $y$ does not satisfy the equivalence $ky\equiv l\bmod{n}$, then $h(\zeta^y)=0$, so we can write $$g(x_m)=p(x_m)\prod_{ky\not\equiv l\bmod{n}}(x_m-\zeta^y)$$  Let $s_i$, for $1\le i\le d$, be the solutions to the above equivalence, then 
\begin{align*}
h(\zeta^{s_i}) &= \lim_{x_m\to \zeta^{s_i}}\frac{\zeta^l}{n}\cdot\frac{x_m^{kn}-1}{x_m^k-\zeta^l}\cdot(1+x_m^{-1}+\cdots+x_m^{-k+1}) \\
&= \lim_{x_m\to \zeta^{s_i}}\frac{\zeta^l}{n}\cdot\frac{x_m^{kn}-1}{x_m^k-\zeta^l}\cdot\frac{x_m^{-k}-1}{x_m^{-1}-1} \\
&= \lim_{x_m\to \zeta^{s_i}}\frac{\zeta^l}{n}\cdot\frac{knx_m^{kn-1}}{kx_m^{k-1}}\cdot\frac{x_m^{-k}-1}{x_m^{-1}-1} \\
&= \frac{\zeta^l}{n}\cdot\frac{kn(\zeta^{s_i})^{kn-1}}{k(\zeta^{s_i})^{k-1}}\cdot\frac{(\zeta^{s_i})^{-k}-1}{(\zeta^{s_i})^{-1}-1} \\
&= \frac{\zeta^l}{n}\cdot\frac{kn\zeta^{-s_i}}{k\zeta^l\zeta^{-s_i}}\cdot\frac{\zeta^{-l}-1}{\zeta^{-s_i}-1} \\
&= \frac{\zeta^{-l}-1}{\zeta^{-s_i}-1}.
\end{align*}
On the other hand,
\begin{align*}
g(\zeta^{s_i}) &= \lim_{x_m\to \zeta^{s_i}} p(x_m)\prod_{ky\not\equiv l\bmod{n}}(x_m-\zeta^y) \\
&= \lim_{x_m\to \zeta^{s_i}} p(x_m)\frac{x_m^n-1}{(x_m-\zeta^{s_1})\cdots(x_m-\zeta^{s_d})} \\
&= \lim_{x_m\to \zeta^{s_i}} p(x_m)\frac{nx_m^{n-1}}{\sum_{j=1}^d\prod_{t\neq j}(x_m-\zeta^{s_t})} \\
&= p(\zeta^{s_i})\frac{n\zeta^{-s_i}}{\prod_{t\neq i}(\zeta^{s_i}-\zeta^{s_t})}.
\end{align*}
We can then solve for $p(\zeta^{s_i})$.
\begin{align*}
h(\zeta^{s_i}) &= g(\zeta^{s_i}) \\
\frac{\zeta^{-l}-1}{\zeta^{-s_i}-1} &= p(\zeta^{s_i})\frac{n\zeta^{-s_i}}{\prod_{t\neq i}(\zeta^{s_i}-\zeta^{s_t})} \\
p(\zeta^{s_i}) &= \frac{\zeta^{s_i}}{n}\prod_{t\neq i}(\zeta^{s_i}-\zeta^{s_t})\frac{\zeta^{-l}-1}{\zeta^{-s_i}-1}.
\end{align*}
Thus, we have an expression for the virtual Adams operations.
$$ \tpsi^k(1_{ml})=\sum_{i=1}^d\frac{\zeta^{-l}-1}{\zeta^{-s_i}-1}1_{ms_i}.$$

Suppose $l\neq0$, and recall equation \eqref{eq.10l} for $\Gamma^{-1}(1_{0l})$.  Set
\begin{align*}
h(x_0) &= \tpsi^k(\Gamma^{-1}(1_{0l})) \\ 
&= \psi^k(\Gamma^{-1}(1_{0l})) \\
&= \frac{\zeta^l}{n(\zeta^l-1)}\cdot(x_0^k-1)\frac{x_0^{kn}-1}{x_0^k-\zeta^l} \\
&= f(x_0)(x_0-1)(x_0^n-1)+g(x_0),
\end{align*}
where the degree of $g$ is at most $n$.  As before, if $d$ does not divide $l$, then $ky\equiv l\bmod{n}$ has no solutions.  Then every $\zeta^y$ is a root of $g$, and 1 is a double root.  This contradicts the degree condition of $g$, so we must have $g(x_0)=0$, and $$\tpsi^k(1_{0l})=0.$$
Assume that $d$ divides $l$, so that $ky\equiv l\bmod{n}$ has $d$ solutions, denoted $s_i$ as above.  If $y$ is not a solution, then $h(\zeta^y)=0$.  In particular, we have $h(1)=0=h'(1)$.  We can then write $$g(x_0)=p(x_0)(x_0-1)\prod_{ky\not\equiv l\bmod{n}}(x_0-\zeta^y).$$  Now assume that $s_i$ is a solution.
\begin{align*} 
h(\zeta^{s_i}) &= \lim_{x_0\to \zeta^{s_i}} \frac{\zeta^l}{n(\zeta^l-1)}\cdot(x_0^k-1)\frac{x_0^{kn}-1}{x_0^k-\zeta^l} \\
&=  \lim_{x_0\to \zeta^{s_i}}\frac{\zeta^l}{n(\zeta^l-1)}\cdot(x_0^k-1)\frac{knx_0^{kn-1}}{kx_0^{k-1}} \\
&= \frac{\zeta^l}{n(\zeta^l-1)}\cdot((\zeta^{s_i})^k-1)\frac{kn(\zeta^{s_i})^{kn-1}}{k(\zeta^{s_i})^{k-1}} \\
&= \frac{\zeta^l}{n(\zeta^l-1)}\cdot(\zeta^l-1)\frac{kn\zeta^{-s_i}}{k\zeta^l\zeta^{-s_i}} \\
&= 1.
\end{align*}
On the other hand,
\begin{align*}
g(\zeta^{s_i}) &= \lim_{x_0\to \zeta^{s_i}}p(x_0)(x_0-1)\prod_{ky\not\equiv l\bmod{n}}(x_0-\zeta^y) \\
&= \lim_{x_0\to \zeta^{s_i}}p(x_0)(x_0-1)\frac{x_0^n-1}{(x_0-\zeta^{s_1})\cdots(x_0-\zeta^{s_d})} \\
&= \lim_{x_0\to \zeta^{s_i}} p(x_0)(x_0-1)\frac{nx_0^{n-1}}{\sum_{j=1}^d\prod_{t\neq j}(x_0-\zeta^{s_t})} \\
&= p(\zeta^{s_i})(\zeta^{s_i}-1)\frac{n\zeta^{-s_i}}{\prod_{t\neq i}(\zeta^{s_i}-\zeta^{s_t})}.
\end{align*}
Since $$h(\zeta^{s_i})=g(\zeta^{s_i}),$$ we obtain $$1=p(\zeta^{s_i})(\zeta^{s_i}-1)\frac{n\zeta^{-s_i}}{\prod_{t\neq i}(\zeta^{s_i}-\zeta^{s_t})},$$ and thus, $$p(\zeta^{s_i}) = \frac{\zeta^{s_i}}{n(\zeta^{s_i}-1)}\prod_{t\neq i}(\zeta^{s_i}-\zeta^{s_t}).$$
We can then conclude that $$\tpsi^k(1_{0l})=\sum_{i=1}^d1_{0s_i}.$$

Equation \eqref{eq.1ml} gives $\Gamma^{-1}(1_{m0})$ for $m\neq 0$.  Define
\begin{align*}
h(x_m) &= \tpsi^k(\Gamma^{-1}(1_{m0})) \\
&= \psi^k(\Gamma^{-1}(1_{m0}))\theta^k(x_m^{-1}) \\
&= \frac{1}{n}\cdot\frac{x_m^{kn}-1}{x_m^{k}-1}\cdot\frac{x_m^{-k}-1}{x_m^{-1}-1} \\
&= f(x_m)(x_m^n-1)+g(x_m).
\end{align*}
for $g$ with degree at most $n-1$.  If $y$ is not a solution to $ky\equiv 0\bmod{n}$, then $h(\zeta^y)=0$ so that $$g(x_m)=p(x_m)\prod_{ky\not\equiv 0\bmod{n}}(x_0-\zeta^y).$$  Further, if $s_i\neq 0$ is a solution to the equivalence, then $((\zeta^y)^{-k}-1)/((\zeta^y)^{-1}-1)=0$, so we have $$g(x_m)=\tilde{p}(x_m)\prod_{y\neq 0}(x_0-\zeta^y)=\tilde{p}(x_m)\frac{x_m^n-1}{x_m-1},$$ for some polynomial $\tilde{p}(x_m)$.  When $y=0$,
\begin{align*}
h(1) &= \lim_{x_m\to 1}\frac{1}{n}\cdot\frac{x_m^{kn}-1}{x_m^{k}-1}\cdot\frac{x_m^{-k}-1}{x_m^{-1}-1} \\
&= \lim_{x_m\to 1}\frac{1}{n}\cdot\frac{knx_m^{kn-1}}{kx_m^{k-1}}\cdot\frac{-kx_m^{-k-1}}{-x_m^{-2}} \\
&= \frac{1}{n}\cdot\frac{kn}{k}\cdot\frac{-k}{-1} \\
&= k.
\end{align*}
Evaluating $g$ at 1 gives
\begin{align*}
g(1) &= \lim_{x_m\to 1}\tilde{p}(x_m)\frac{x_m^n-1}{x_m-1} \\
&= \lim_{x_m\to1}\tilde{p}(x_m)\frac{nx_m^{n-1}}{1} \\
&= \tilde{p}(1)n.
\end{align*}
Therefore, we have $g(x_m)=\frac{k}{n}\frac{x_m^n-1}{x_m-1}=1_{m0}$, and so $$\tpsi^k(1_{m0})=k1_{m0}.$$

The polynomial defining $\Gamma^{-1}(1_{00})$ is given by equation \eqref{eq.100}.  The Adams operations are 
\begin{align*}
h(x_0) &= \tpsi^k(\Gamma^{-1}(1_{00})) \\
&= \psi^k(\Gamma^{-1}(1_{00})) \\
&= \frac{1}{2n}((1-n)x_0^k+(1+n))\frac{x_0^{kn}-1}{x_0^k-1} \\
&= f(x_0)(x_0-1)(x_0^n-1)+g(x_0).
\end{align*} 
where $g$ is a polynomial of degree at most $n$.  As above, if $y$ is not a solution to $ky\equiv 0\bmod{n}$, then $h(\zeta^y)=0$, and we can write $$g(x_0)=p(x_0)\prod_{ky\not\equiv 0\bmod{n}}(x_0-\zeta^y)=p(x_0)\frac{x_0^n-1}{(x_0-\zeta^{s_1})\cdots(x_0-\zeta^{s_d})}.$$  Now suppose $s_i$ is a solution.
\begin{align*}
h(\zeta^{s_i}) &= \lim_{x_0\to\zeta^{s_i}} \frac{1}{2n}((1-n)x_0^k+(1+n))\frac{x_0^{kn}-1}{x_0^k-1} \\
&= \lim_{x_0\to\zeta^{s_i}} \frac{1}{2n}((1-n)x_0^k+(1+n))\frac{knx_0^{kn-1}}{kx_0^{k-1}} \\
&=  \frac{1}{2n}((1-n)(\zeta^{s_i})^k+(1+n))\frac{kn(\zeta^{s_i})^{kn-1}}{k(\zeta^{s_i})^{k-1}} \\
&= \frac{1}{2n}((1-n)+(1+n))\frac{kn\zeta^{-s_i}}{k\zeta^{-s_i}} \\
&= 1.
\end{align*}
Evaluating $g$ gives
\begin{align*}
g(\zeta^{s_i}) &= \lim_{x_0\to\zeta^{s_i}}p(x_0)\frac{x_0^n-1}{(x_0-\zeta^{s_1})\cdots(x_0-\zeta^{s_d})} \\
&= \lim_{x_0\to\zeta^{s_i}}p(x_0)\frac{nx_0^{n-1}}{\sum_{j=1}^d\prod_{t\neq j}(x_0-\zeta^{s_t})} \\
&= p(\zeta^{s_i})\frac{n\zeta^{-s_i}}{\prod_{t\neq i}(\zeta^{s_i}-\zeta^{s_t})}.
\end{align*}
Setting these expressions equal and solving for $p(\zeta^{s_i})$ yields $$p(\zeta^{s_i})=\frac{\zeta^{s_i}}{n}\prod_{t\neq i}(\zeta^{s_i}-\zeta^{s_t}).$$  Since it is possible for 1 to be a double root, we take derivatives and evaluate at 1.
\begin{align*}
h'(x_0) &= \frac{1}{2n}(k(1-n)x_0^{k-1})\frac{x_0^{kn}-1}{x_0^k-1}\\ & \quad +\frac{1}{2n}((1-n)x_0^k+(1+n))\frac{(x_0^k-1)knx_0^{kn-1}-(x_0^{kn}-1)kx_0^{k-1}}{(x_0^k-1)^2} \\
h'(1) &= \frac{1}{2n}k(1-n)\frac{kn}{k}+\frac{1}{2n}(2)\frac{kn(n-1)}{2} \\
&= 0.
\end{align*}
Recall that $s_1=0$ in this case.
\begin{align*}
g'(x_0) &= p'(x_0)\frac{x_0^n-1}{\prod_{t=1}^d(x_0-\zeta^{s_t})}\\ & \quad +p(x_0)\frac{nx_0^{n-1}\prod_{t=1}^d(x_0-\zeta^{s_t})-(x_0^n-1)\sum_{j=1}^d\prod_{t\neq j}(x_0-\zeta^{s_t})}{\prod_{t=1}^d(x_0-\zeta^{s_t})^2} \\
g'(1) &= \lim_{x_0\to1}\Big[p'(x_0)\frac{x_0^n-1}{\prod_{t=1}^d(x_0-\zeta^{s_t})}\\ & \quad +p(x_0)\frac{nx_0^{n-1}\prod_{t=1}^d(x_0-\zeta^{s_t})-(x_0^n-1)\sum_{j=1}^d\prod_{t\neq j}(x_0-\zeta^{s_t})}{\prod_{t=1}^d(x_0-\zeta^{s_t})^2}\Big] \\
&= p'(1)\frac{n}{\prod_{t=2}^d(1-\zeta^{s_t})}\\ & \quad +\frac{1}{n}\Big[\prod_{t=2}^d(1-\zeta^{s_t})\Big]\frac{n(n-1)\prod_{t=2}^d(1-\zeta^{s_t})-2n\sum_{j=2}^d\prod_{t\neq j\,\textup{or}\,1}(1-\zeta^{s_t})}{2\prod_{t=2}^d(1-\zeta^{s_t})^2}.
\end{align*}
So we have
\begin{align*} 0 &= p'(1)\frac{n}{\prod_{t=2}^d(1-\zeta^{s_t})}\\ & \quad +\frac{1}{n}\Big[\prod_{t=2}^d(1-\zeta^{s_t})\Big]\frac{n(n-1)\prod_{t=2}^d(1-\zeta^{s_t})-2n\sum_{j=2}^d\prod_{t\neq j\,\textup{or}\,1}(1-\zeta^{s_t})}{2\prod_{t=2}^d(1-\zeta^{s_t})^2} \\
0 &= 2np'(1)+(n-1)\prod_{t=2}^d(1-\zeta^{s_t})-2\sum_{j=2}^d\prod_{t\neq j\,\textup{or}\,1}(1-\zeta^{s_t}) \\
p'(1) &= \frac{1}{2n}\Big(2\sum_{j=2}^d\prod_{t\neq j\,\textup{or}\,1}(1-\zeta^{s_t})-(n-1)\prod_{t=2}^d(1-\zeta^{s_t})\Big).
\end{align*}
Consider the sum
\begin{align*}
r(x_0) &= \sum_{i=1}^d\Gamma^{-1}(1_{0s_i}) \\
&= \Gamma^{-1}(1_{00})+\sum_{i=2}^d\Gamma^{-1}(1_{0s_i}) \\
&= \frac{1}{2n}((1-n)x_0+(1+n))\frac{x_0^n-1}{x_0-1}+\sum_{i=2}^d\frac{\zeta^{s_i}}{n(\zeta^{s_i}-1)}(x_0-1)\frac{x_0^n-1}{x_0-\zeta^{s_i}}.
\end{align*}
Notice that $r(\zeta^y)$ agrees with the expression above for all values of $y$ and has $r'(1)=h'(1)$, so by taking $\Gamma$ of each side, we obtain $$\tpsi^k(1_{00})=\sum_{i=1}^d1_{0s_i}.$$

Lastly, we consider the Adams operations on $x_{00}$.  Recall that $\Gamma^{-1}(x_{00})$ is given in equation \eqref{eq.x00}.  Let
\begin{align*}
h(x_0) &= \tpsi^k(\Gamma^{-1}(x_{00})) \\
&= \psi^k(\Gamma^{-1}(x_{00})) \\
&= \frac{1}{2n}((3-n)x_0^k+(n-1))\frac{x_0^{kn}-1}{x_0^k-1} \\
&= f(x_0)(x_0-1)(x_0^n-1)+g(x_0).
\end{align*}
where the degree of $g$ is at most $n$.  When $y$ is not a solution to $ky\equiv 0\bmod{n}$, then $h(\zeta^y)=0$, so that $$g(x_0)=p(x_0)\prod_{ky\not\equiv 0\bmod{n}}(x_0-\zeta^y)=p(x_0)\frac{x_0^n-1}{\prod_{t=1}^d(x_0-\zeta^{s_t})}.$$  Assume $s_i$ is a solution to the equivalence $ky\equiv 0\bmod{n}$.
\begin{align*}
h(\zeta^{s_i}) &= \lim_{x_0\to\zeta^{s_i}} \frac{1}{2n}((1-n)x_0^k+(1+n))\frac{x_0^{kn}-1}{x_0^k-1} \\
&= \lim_{x_0\to\zeta^{s_i}} \frac{1}{2n}((3-n)x_0^k+(n-1))\frac{knx_0^{kn-1}}{kx_0^{k-1}} \\
&=  \frac{1}{2n}((3-n)(\zeta^{s_i})^k+(n-1))\frac{kn(\zeta^{s_i})^{kn-1}}{k(\zeta^{s_i})^{k-1}} \\
&= \frac{1}{2n}((3-n)+(n-1))\frac{kn\zeta^{-s_i}}{k\zeta^{-s_i}} \\
&= 1 \\
g(\zeta^{s_i}) &= \lim_{x_0\to\zeta^{s_i}}p(x_0)\frac{x_0^n-1}{\prod_{t=1}^d(x_0-\zeta^{s_t})} \\
&= \lim_{x_0\to\zeta^{s_i}}p(x_0)\frac{nx_0^{n-1}}{\sum_{j=1}^d\prod_{t\neq j}(x_0-\zeta^{s_t})} \\
&= p(\zeta^{s_i})\frac{n\zeta^{-s_i}}{\prod_{t\neq i}(\zeta^{s_i}-\zeta^{s_t})}.
\end{align*}
Equating these two expressions and solving for $p(\zeta^{s_i})$ gives $$p(\zeta^{s_i})=\frac{\zeta^{s_i}}{n}\prod_{t\neq i}(\zeta^{s_i}-\zeta^{s_t}).$$  As in the case of $1_{00}$, 1 may be a double root, so we take derivatives and evaluate to obtain
\begin{align*}
h'(x_0) &= \frac{1}{2n}(k(3-n)x_0^{k-1})\frac{x_0^{kn}-1}{x_0^k-1}\\ & \quad +\frac{1}{2n}((3-n)x_0^k+(n-1))\frac{(x_0^k-1)knx_0^{kn-1}-(x_0^{kn}-1)kx_0^{k-1}}{(x_0^k-1)^2} \\
h'(1) &= \frac{1}{2n}k(3-n)\frac{kn}{k}+\frac{1}{2n}(2)\frac{kn(n-1)}{2} \\
&= k.
\end{align*}
The same calculation as for $1_{00}$ holds for $g'(1)$, so we find that $$p'(1)=\frac{1}{2n}\Big(2k+2\sum_{j=2}^d\prod_{t\neq j\,\textup{or}\,1}(1-\zeta^{s_t})-(n-1)\prod_{t=2}^d(1-\zeta^{s_t})\Big).$$
Consider the expression
\begin{align*}
r(x_0) &= k\Gamma^{-1}(x_{00})-(k-1)\Gamma^{-1}(1_{00})+\sum_{i=2}^d\Gamma^{-1}(1_{0s_i}) \\
&= k\frac{1}{2n}((3-n)x_0+(n-1))\frac{x_0^n-1}{x_0-1}+(k-1)\frac{1}{2n}((1-n)x_0+(1+n))\frac{x_0^n-1}{x_0-1}\\ & \quad +\sum_{i=2}^d\frac{\zeta^{s_i}}{n(\zeta^{s_i}-1)}(x_0-1)\frac{x_0^n-1}{x_0-\zeta^{s_i}}.
\end{align*}
Since $r(\zeta^y)=h(\zeta^y)$ for all $y$, and $r'(1)=h'(1)$, we must have $r(x_0)=h(x_0)$.  Taking $\Gamma$ of both sides gives $$\tpsi^k(x_{00})=kx_{00}-(k-1)1_{00}+\sum_{i=2}^d1_{0s_i}.$$
\end{proof}

It will be useful to renormalize the generators of $K_l$ for $l\neq 0$ and $m\neq 0$ to make the virtual multiplication and virtual Adams operations more concise.  
\begin{df} Let $$\hat{1}_{ml}:=\frac{1}{1-\zeta^{-l}}1_{ml}$$ for $l\neq 0$ and $m\neq 0$.  When $l=0$, let $\hat{1}_{m0}=1_{m0}$ for all $m$.  Similarly, let $\hat{1}_{0l}=1_{0l}$ for all $l$.
\end{df}
With these generators, it is immediate that whenever $l\neq 0$, the localization $\mathcal{K}_l$ is isomorphic as a ring to the group ring $\mathbb{C}[\mathbb{Z}/n\mathbb{Z}]$.  Since the group ring has a semisimple basis, we can construct a semsimple basis for $\mathcal{K}_l$ as follows.
\begin{df} For $l\neq 0$, set $$u_l^q=\frac{1}{n}\sum_{i=0}^{n-1}\zeta^{-iq}\hat{1}_{il}$$ for $q=0,\ldots,n-1$.  Define $u_0^0=x_{00}-1_{00}$ and $u_0^m=1_{m0}$ for $m\neq 0$.
\end{df}
With these generators, the virtual multiplication becomes $$u_{l_1}^{q_1}*u_{l_2}^{q_2}=\delta_{l_1l_2}\delta_{q_1q_2}u_{l_1}^{q_1},$$ where $\delta$ is the Kronecker delta.
Note that the identity element $\mathbf{1}\in\mathcal{K}$ is $$\mathbf{1}=\sum_{l=0}^{n-1}1_{0l}.$$  Given these generators with the virtual multiplication, $K(\pn)_{\mathbb{C}}$ has the following presentation: $$\mathbb{C}[u_l^q,1_{0l}]_{l,q=0}^{n-1}/I.$$
where the ideal $I$ is generated by the polynomials below.
\begin{align*} u_0^qu_0^{q'} & & u_l^qu_{l'}^{q'}-\delta_{ll'}\delta_{qq'}u_l^q, & \hspace{.1in} \textup{if $l$ and $l'$ are not both 0} \\ \sum_{l=0}^{n-1}1_{0l}-\mathbf{1} & & \sum_{q=0}^{n-1}u_l^q-1_{0l}, & \hspace{.1in} \textup{if}\,\,\,l\neq 0 \\ 1_{0l}1_{0l'}-\delta_{ll'}1_{0l}, & \hspace{.1in}\textup{for $l=0,\ldots,n-1$} & u_l^q1_{0l'}-\delta_{ll'}u_l^q, & \hspace{.1in}\textup{for $l=0,\ldots,n-1$}.
\end{align*} 

Proposition 6.1 can be restated in terms of these new generators.
\begin{prop}
With respect to the generators $u_l^q$ on the localization, the virtual Adams operations have the following forms for all $k\ge 1$.
\begin{align*} \widetilde{\psi}^k(u_0^q) &= ku_0^q, \hspace{.1in} \textup{if}\,\,q=0,\ldots, n-1 \\ \widetilde{\psi}^k(u_l^q) &= \begin{cases} \sum_{i=1}^du_{s_i}^q, \hspace{.1in} \textup{for all}\,\,q=0,\ldots,n-1\,\,\textup{if}\,\, d\,|\,l\,\,\textup{and}\,\,l\neq 0 \\ 
0, \hspace{.1in} \textup{if}\,\,d\nmid l \end{cases} \\
\tpsi^k(\mathbf{1}) &= \mathbf{1}.
\end{align*}
\end{prop}

\section{The virtual line elements}\label{sec:virlineelts}
Consider the augmented virtual K-ring $(K(I\pn)_{\mathbb{C}},*,\mathbf{1},\widetilde{\psi}^k,\widetilde{\epsilon})$ with $\mathbb{C}$-linear extensions of $\widetilde{\psi}^k$ and $\widetilde{\epsilon}$.
\begin{prop} The group of virtual line elements, $\mathcal{P}$, is isomorphic to $(\mathbb{Z}/n\mathbb{Z})^n\times\mathbb{C}^n$ via the isomorphism $$\Phi:(\mathbb{Z}/n\mathbb{Z})^n\times\mathbb{C}^n\longrightarrow\mathcal{P}$$ $$(f_0,\ldots,f_{n-1};\beta^0_0,\ldots,\beta_0^{n-1})\mapsto \mathbf{1}+\sum_{q=0}^{n-1}\sum_{l=1}^{n-1}(\zeta^{lf_q}-1)u_l^q+\sum_{q=0}^{n-1}\beta_0^qu_0^q.$$
\end{prop}
\begin{proof}
Every element in $K(\pn)_{\mathbb{C}}$ can be written as $$\alpha\id+\sum_{l,q=0}^{n-1}\beta_l^qu_l^q$$  We would like to find the virtual line elements with respect to these generators.  Recall that the virtual line elements are determined by the equation $$\stack{L}^k=\widetilde{\psi}^k(\stack{L})$$ with respect to the virtual multiplication and Adams operations, for all positive integers $k$.  First, consider the case where $\alpha=0$, so that $\stack{L}=\sum_{l,q=0}^{n-1}\beta_l^qu_l^q$.  Then if $k=n$, we have $$\widetilde{\psi}^n(\stack{L})=\sum_{q=0}^{n-1}n\beta_0^qu_0^q$$ and by the relations above, $$\stack{L}^n=\sum_{l=1}^{n-1}\sum_{q=0}^{n-1}(\beta_l^q)^nu_l^q.$$  Thus, it is immediate that $\beta_l^q=0$ for all $l$ and $q$, which contradicts the invertibility of $\stack{L}$.  Thus, we may assume that $\alpha\neq 0$.  Write $\widetilde{\beta}_l^q=\beta_l^q/\alpha$ and $l=\frac{n}{d}b+r$ for some $b$ and $r$, and $d=\gcd(k,n)$.  In this case, 
\begin{align*} 
\widetilde{\psi}^k(\stack{L}) &= \alpha\mathbf{1}+\sum_{q=0}^{n-1}\sum_{i=1}^{d}\beta_l^qu_{s_i}^q+\sum_{q=0}^{p-1}\beta_0^qku_0^q \\
&= \alpha\mathbf{1}+\sum_{q=0}^{n-1}\sum_{r=0}^{\frac{n}{d}-1}\sum_{b=0}^{d-1}\beta_{rk\bmod{n}}^qu_l^q+\sum_{q=0}^{n-1}\beta_0^qku_0^q,
\end{align*}
for values of $k\not\equiv 0\bmod{n}$.  If $k\equiv 0\bmod{n}$, then $$\widetilde{\psi}^k(\stack{L})=\alpha\mathbf{1}+\sum_{q=0}^{n-1}\beta_0^qku_0^q.$$
On the other hand,
\begin{align*} \stack{L}^k &= \Big(\alpha\id+\sum_{l,q=0}^{n-1}\beta_l^qu_l^q\Big)^k \\ &= \alpha^k\Big(\mathbf{1}+\sum_{l=1}^{n-1}\sum_{q=0}^{n-1}\widetilde{\beta}_l^qu_l^q+\sum_{q=0}^{n-1}\widetilde{\beta}_0^qu_0^q\Big)^k \\ &= \alpha^k\Big(\mathbf{1}+\sum_{l=1}^{n-1}\sum_{q=0}^{n-1}((1+\widetilde{\beta}_l^q)^k-1)u_l^q+k\sum_{q=0}^{n-1}\widetilde{\beta}_0^q\Big).
\end{align*}
In order for equation \eqref{eq.inpsi} to be satisfied, the following three conditions must hold
\begin{align*} \alpha^k &= \alpha, & \alpha^k((1+\widetilde{\beta}_l^q)^k-1) &= \beta_{rk\bmod{n}}^q, & \alpha^kk\widetilde{\beta}_0^q=k\beta_0^q. \end{align*}
Since these equations must hold for all positive integers $k$, the first condition implies that $\alpha=1$.  Thus, the second and third conditions reduce to
\begin{align*} (1+\beta_l^q)^k-1 &= \beta_{rk\bmod{n}}^q, & k\beta_0^q &= k\beta_0^q. \end{align*}
The third condition is satisfied for any choice of $\beta_0^q\in\mathbb{C}$.  If $r=1$ and $k=n$, then the second equation states that 
\begin{align*} (1+\beta_{nb/d+1}^q)^n &= 0 \\  \beta_{nb/d+1}^q &= \zeta^{f_q}-1, \end{align*}
for some $f_q=0,\ldots,n-1$.
Further, each $\beta_k^q$ is determined by $\beta_{nb/d+1}^q$ as follows:
\begin{align*} (1+\beta_{nb/d+1}^q)^k-1 &= \beta_k^q \\ 
(\zeta^{f_q})^k-1 &= \beta_k^q \\ 
\zeta^{kf_q}-1 &= \beta_k^q.
\end{align*}
Therefore, a general line element in terms of the generators above has the form $$\stack{L}=\mathbf{1}+\sum_{q=0}^{n-1}\sum_{l=1}^{n-1}(\zeta^{lf_q}-1)u_l^q+\sum_{q=0}^{n-1}\beta_0^qu_0^q,$$ where $f_q\in\{0,\ldots,n-1\}$ and $\beta_0^q\in\mathbb{C}$.  Write this virtual line element as $\stack{L}(f_0,\ldots,f_{n-1};\beta_0^0,\ldots,\beta_0^{n-1})$.  Let $\stack{L}$ and $\stack{L}'$ be virtual line elements, then
\begin{align*} 
\stack{L}\stack{L}' &=  \mathbf{1}+\sum_{q=0}^{n-1}\sum_{l=1}^{n-1}((\zeta^{lf_q}-1)+(\zeta^{lf_q}-1)(\zeta^{lf'_q}-1)+(\zeta^{lf'_q}-1))u_l^q+\sum_{q=0}^{n-1}(\beta_0^q+{\beta'}_0^q)u_0^q \\
&=  \mathbf{1}+\sum_{q=0}^{n-1}\sum_{l=1}^{n-1}(\zeta^{l(f_q+f_q')}-1)u_l^q+\sum_{q=0}^{n-1}(\beta_0^q+{\beta'}_0^q)u_0^q.
\end{align*}
Thus, multiplication of virtual line elements has the following property: 
\begin{equation} \label{eq.linemult} \stack{L}\stack{L}'=\stack{L}(f_0+f'_0,\ldots,f_{n-1}+f'_{n-1};\beta_0^0+{\beta'}_0^0,\ldots,\beta_0^{n-1}+{\beta'}_0^{n-1}).\end{equation}
Note that this implies that $\stack{L}(f_0,\ldots,f_{p-1};\beta_0^0,\ldots,\beta_0^{p-1})^{-1}=\stack{L}(-f_0,\ldots,-f_{p-1};-\beta_0^0,\ldots,-\beta_0^{n-1})$.  Thus, the map $\Phi$ as defined above is an isomorphism of groups.
\end{proof}

\begin{prop} The virtual line elements $\mathcal{P}$ span $K(I\pn)_{\mathbb{C}}$ as a complex vector space.
\end{prop}
\begin{proof}  We show that every generator $u_l^q$ and $\mathbf{1}$ can be written as a linear combination of line elements.  Note that $\mathbf{1}=\stack{L}(0,\ldots,0;0,\ldots,0)$.  By subtracting the constant term from the line element $\stack{L}(0,\ldots,0;0,\ldots,0,1,0,\ldots,0)$, we obtain $u_0^q$.  Let $M$ be the $n(n-1)\times n^n$ matrix whose rows are given by pairs $(l,q)$ and columns by $n$-tuples $(f_0,\ldots,f_{n-1})$, where an entry of $M$ is $(\zeta^{lf_q}-1)$.  Construct a submatrix $A$ of $M$ whose columns are the columns of $M$ determined by the $n$-tuples $(\alpha,0,\ldots,0),\,(0,\alpha,0,\ldots,0),\ldots,\,(0,\ldots,0,\alpha)$ for each $\alpha\in\mathbb{Z}/n\mathbb{Z}$.  The matrix $A$ is a block diagonal matrix, whose diagonal blocks, $B$, are the $(n-1)\times(n-1)$ matrices 
$$B=\begin{pmatrix}
\zeta-1 & \zeta^2-1 & \cdots & \zeta^{n-1}-1 \\
\zeta^2-1 & \zeta^{2(2)}-1 & \cdots & \zeta^{2(n-1)}-1 \\
\vdots & \vdots & \ddots & \vdots \\
\zeta^{n-1}-1 & \zeta^{2(n-1)}-1 & \cdots & \zeta^{(n-1)(n-1)}-1
\end{pmatrix}.$$
Squaring the matrix $B$ yields the real matrix
$$B^2=\begin{pmatrix} n & n & \cdots & 2n \\
\vdots & \vdots & \iddots & \vdots \\
n & 2n & \cdots & n \\
2n & n & \cdots & n
\end{pmatrix}.$$
This matrix has a linearly independent set of $n-1$ eigenvectors $$\begin{pmatrix} 1 \\ 1 \\ 1 \\ \vdots \\ 1 \end{pmatrix}, \hspace{.1in} \begin{pmatrix} 1\\ -1 \\ 0 \\ \vdots \\ 0 \end{pmatrix}, \hspace{.1in} \begin{pmatrix} 1\\0\\-1\\ \vdots\\0\end{pmatrix}, \cdots, \begin{pmatrix} 1\\0\\0\\ \vdots\\-1\end{pmatrix}.$$
As $A$ is block diagonal, $A^2$ is also block diagonal with diagonal block entries equal to $B^2$.  Thus, $A$ has a linearly independent set of $n(n-1)$ eigenvectors given by placing the eigenvectors above in a column vector with the rest of the entries equal to zero.  Thus, the rank of $A^2$ is $n(n-1)$, and so the determinant is nonzero.  This implies the determinant of $A$ is nonzero, and so we conclude that the rank of $A$ is $n(n-1)$.  Therefore, the rank of $M$ is $n(n-1)$ as well.
\end{proof}

\section{Comparing the virtual K-theory of $\pn$ with the K-theory of its crepant resolution}
To obtain the result from \cite{EJK2} for $\pn$, we need to show that the localization ${\mathcal{K}_0}_{\mathbb{C}}$ with its virtual product is isomorphic as a ring to the K-theory of a resolution of the cotangent bundle of $\pn$ with the ordinary product.  Recall that the K-theory ring of the resolution tensored with $\mathbb{C}$ is isomorphic as a ring to $$K(Z_n)_{\mathbb{C}}=\frac{\mathbb{C}[\chi_0^{\pm1},\ldots,\chi_{n-1}^{\pm1}]}{\langle(\chi_i-1)(\chi_j-1)\rangle},$$ where $Z_n$ is the resolution of \cite{EJK2}, discussed in Section 3 of this paper.

There are natural generators for $\mathcal{K}$ given by the line elements 
\begin{align*} 
\sigma_i &= \stack{L}(0,\ldots, 0,1,0,\ldots,0;0,\ldots,0) \\ 
\nu_j &= \mathbf{1}+\delta_j \hspace{.1in} \textup{where} \\
\delta_j &= \stack{L}(0,\ldots,0;0,\ldots,0,1,0,\ldots,0)-\stack{L}(0,\ldots,0;0,\ldots,0) \\ &= \stack{L}(0,\ldots,0;0,\ldots,0,1,0,\ldots,0)-\mathbf{1}, \end{align*}
where the 1 is in the $i$th and $j$th slots, respectively.  Note that in the proof of Proposition 7.2, we show that the line elements are spanned as a vector space by $\stack{L}(0,\ldots, 0,\alpha,0,\ldots,0;0,\ldots,0)$ and  $\stack{L}(0,\ldots,0;0,\ldots,0,\alpha,0,\ldots,0)$, for $\alpha\in\mathbb{Z}/n\mathbb{Z}$.  Since multiplication of line elements is additive in the entries of the $\stack{L}$'s, we see that $\sigma_i$ and $\nu_j$ indeed generate $\mathcal{K}$ as a ring.
\begin{thm} The virtual K-theory ring $K(I\pn)_\mathbb{C}$ admits the following presentation: $$\mathbb{C}[\sigma_0^\pm,\ldots,\sigma_{n-1}^\pm,\nu_0^\pm,\ldots,\nu_{n-1}^\pm]/I$$ where the ideal $I$ is generated by the polynomials
\begin{align*} \sigma_i^n-\mathbf{1} &  & (\nu_i-\mathbf{1})(\nu_j-\mathbf{1}) & \\ \sigma_i(\nu_j-\mathbf{1})-(\nu_j-\mathbf{1}) & & (\sigma_i-\mathbf{1})(\sigma_j-\mathbf{1}), & \hspace{.1in}\textup{if}\,\, i\neq j,  \end{align*}
for the virtual line elements $\sigma_i$ and $\nu_j$ defined above.
\end{thm}
\begin{proof}  The virtual line elements $\sigma_i$ and $\nu_j$ generate $\mathcal{P}$ by the multiplication property \eqref{eq.linemult} of virtual line elements.  Applying this property again, the relations in the ideal $I$ follow easily.
\end{proof}

We will now use Theorem 8.1 to find a presentation for the localization $\mathcal{K}_0$ with complex coefficients.  This will be isomorphic as a $\psi$-ring to the ordinary K-theory of the resolution $Z_n$, and by establishing an isomorphism as $\psi$-rings between ${\mathcal{K}_0}_{\mathbb{C}}$ and the virtual augmentation completion $\widehat{K}(I\pn)_{\mathbb{C}}$ we will have a generalization of Proposition 7.22 of \cite{EJK2}.

By the definition of the virtual line elements $\stack{L}$, the map $\Gamma_0:=\Gamma|_{\mathcal{K}_0}$ takes 
$$\begin{xymatrix}{\mathcal{K} \ar[dd]^{\Gamma_0} \\ \\ \mathcal{K}_0}\end{xymatrix} \hspace{.2in} \begin{xymatrix}{\mathbf{1}+\displaystyle\sum_{l=1}^{p-1}\sum_{q=0}^{p-1}(\zeta^{lf_q}-1)u_l^q+\displaystyle\sum_{q=0}^{p-1}\beta_0^qu_0^q \ar@{|->}[d] \\ 1_{00}+\displaystyle\sum_{q=0}^{p-1}\beta_0^qu_0^q}
\end{xymatrix}$$
Note that this implies that all the generators $\sigma_i$ will be mapped to the identity $1_{00}$.  In terms of the presentation of $\mathcal{K}$ above, we can interpret $\Gamma_0$ by the canonical projection
$$\begin{xymatrix} {{\mathcal{K}}_{\mathbb{C}}\cong\mathbb{C}[\sigma_0^{\pm1},\ldots,\sigma_{n-1}^{\pm1},\nu_0^{\pm1},\ldots,\nu_{n-1}^{\pm1}]/I \ar[d] \\ {\mathcal{K}_0}_{\mathbb{C}}\cong\mathbb{C}[\sigma_0^{\pm1},\ldots,\sigma_{n-1}^{\pm1},\nu_0^{\pm1},\ldots,\nu_{n-1}^{\pm1}]/I'}\end{xymatrix}$$
where $I'$ is generated by the same polynomials as $I$ with the addition of the polynomials $\sigma_i-\mathbf{1}$.  Further, the target ring is isomorphic as a ring to $$\frac{\mathbb{C}[\widehat{\nu}_0^{\pm1},\ldots,\widehat{\nu}_{n-1}^{\pm1}]}{\langle(\widehat{\nu}_i-1)(\widehat{\nu}_j-1)\rangle}$$
via a ring isomorphism sending $\nu_i\mapsto\widehat{\nu}_i$ and $\sigma_j\mapsto 1$.  In the proof of Proposition 7.22 of \cite{EJK2}, it is noted that since the virtual K-theory is the quotient of the coordinate ring of a torus, the virtual augmentation completion $\widehat{K}(I\pn)_{\mathbb{C}}$ is the localization of $K(I\pn)_{\mathbb{C}}$ at the maximal ideal corresponding to the identity of the torus.  But this is exactly the localization $\mathcal{K}_0$.  Thus, we have that $$\widehat{K}(I\pn)_{\mathbb{C}}\cong \frac{\mathbb{C}[\widehat{\nu}_0^{\pm1},\ldots,\widehat{\nu}_{n-1}^{\pm1}]}{\langle(\widehat{\nu}_i-1)(\widehat{\nu}_j-1)\rangle}.$$  This is clearly isomorphic to the K-theory of the resolution, so we have proved the following theorem.
\begin{thm} For all integers $n\ge 2$, there is an isomorphism of $\psi$-rings $$\widehat{K}(I\pn)_{\mathbb{C}}\cong K(Z_n)_\mathbb{C},$$ with respect to the virtual product on the virtual augmentation completion and the ordinary product on the K-theory of $Z_n$, where $Z_n$ is the toric crepant resolution of $\pn$ from \cite{EJK}.
\end{thm}


\nocite{JKK2, CR, EJK0, ARZ}

\bibliographystyle{plain}
\bibliography{Bib}

\begin{thebibliography}{1}

\bibitem{ARZ}
Alejandro Adem, Yongbin Ruan, and Bin Zhang.
\newblock {A stringy product on twisted orbifold K-theory}.
\newblock {\em arXiv.org}, math.AT, May 2006.

\bibitem{CR}
Weimin Chen and Yongbin Ruan.
\newblock {A new cohomology theory of orbifold}.
\newblock {\em Communications in Mathematical Physics}, 248(1):1--31, 2004.

\bibitem{EG}
Dan Edidin and William Graham.
\newblock {Nonabelian localization in equivariant K-theory and Riemann-Roch for
  quotients}.
\newblock {\em Advances in Mathematics}, 198(2):547--582, 2005.

\bibitem{EJK0}
Dan Edidin, Tyler~J Jarvis, and Takashi Kimura.
\newblock {Logarithmic trace and orbifold products}.
\newblock {\em arXiv.org}, math.AG, April 2009.

\bibitem{EJK1}
Dan Edidin, Tyler~J Jarvis, and Takashi Kimura.
\newblock {A plethora of inertial products}.
\newblock {\em arXiv.org}, math.AG, September 2012.

\bibitem{EJK2}
Dan Edidin, Tyler~J Jarvis, and Takashi Kimura.
\newblock {Chern classes and compatible power operations in inertial K-theory}.
\newblock {\em arXiv.org}, math.AG, September 2012.

\bibitem{EJK}
Dan Edidin, Tyler~J. Jarvis, and Takashi Kimura.
\newblock Chern classes and compatible power operations in inertial k-theory.
\newblock {\em Preprint}, 2012.

\bibitem{GLSUX}
A.~Gonzalez, E.~Lupercio, C.~Segovia, B.~Uribe, and M.A. Xicotencatl.
\newblock {Chen-Ruan cohomology of cotangent orbifolds and Chas-Sullivan string
  topology}.
\newblock {\em Arxiv preprint math/0610899}, 2006.

\bibitem{JKK2}
Tyler~J Jarvis, Ralph Kaufmann, and Takashi Kimura.
\newblock {Stringy K-theory and the Chern character}.
\newblock {\em arXiv.org}, math.AG, February 2005.

\end{thebibliography}

\end{document}